\documentclass[12pt]{amsart}%
\usepackage{amsmath,amsfonts,amsthm,amscd,amssymb,stmaryrd,mathrsfs}
\usepackage{amsmath}
\usepackage{amsfonts}
\usepackage{amssymb}
\usepackage{verbatim}
\usepackage{graphicx}
\usepackage{color}
\usepackage{caption}%
\providecommand{\U}[1]{\protect\rule{.1in}{.1in}}
\newtheorem{theorem}{Theorem}

\newtheorem{corollary}[theorem]{Corollary}

\newtheorem{definition}[theorem]{Definition}

\newtheorem{lemma}[theorem]{Lemma}
\newtheorem{notation}[theorem]{Notation}

\newtheorem{proposition}[theorem]{Proposition}
\newtheorem{remark}[theorem]{Remark}

\newtheorem{convention}[theorem]{Convention}

\numberwithin{equation}{section}
\numberwithin{theorem}{section}
\newtheorem{thmx}{Theorem}

\newtheorem{corx}{Corollary}

\begin{document}
\title{ Ricci Curvature and the Manifold Learning Problem}
\author{Antonio G. Ache}
\address{Mathematics Department, University of Notre Dame, 
Notre Dame, Indiana, 46556 USA }
\email{aache@nd.edu}
\author{Micah W. Warren}
\address{Department of Mathematics, University of Oregon, Eugene OR 97403}
\email{micahw@oregon.edu}
\thanks{The first author was partially supported by a postdoctoral fellowship of the
National Science Foundation, award No. DMS-1204742.}
\thanks{The second author was partially supported by NSF Grant DMS-1438359. }
\maketitle  

\begin{abstract}
{ Consider a sample of $n$ points taken i.i.d from a submanifold $\Sigma$ of
Euclidean space. We show that there is a way to estimate the Ricci curvature
of $\Sigma$ with respect to the induced metric from the sample. Our method is
grounded in the notions of \emph{Carr\'{e} du Champ} for diffusion
semi-groups, the theory of empirical processes and local Principal Component
Analysis.}

\end{abstract}
\tableofcontents


\section{Introduction}

In this paper we are concerned with the structure of sets of large data in
high dimensions. Even though we deal with sets of points in $\mathbb{R}^{N}$
for $N$ large, a common assumption when studying large data sets is that the
points lie in or on the vicinity of an embedded low dimensional submanifold
$\Sigma^{d}$ of $\mathbb{R}^{N}$. This assumption is oftentimes called
\emph{the manifold assumption} and the study of geometric and topological
properties of sets satisfying the manifold assumption is what we call nowadays
\emph{manifold learning}. The interest in the structure of large data sets
comes from the need of organizing information arising from many different
sources, for example, images, signals, genomes and other outcomes. Even though
there has been significant progress in the manifold learning problem in the
last decade, a fundamental question remains unanswered: construct an algorithm
for \emph{learning} or effectively estimating the curvature of a manifold that
is being approximated by a point cloud.{ In this paper we lay the theoretical
foundation for estimating the Ricci curvature of an embedded submanifold of
$\mathbb{R}^{N}$ if one only knows a point cloud approximating the
submanifold. In particular, combining with the recent work of Singer and Wu
\cite{SW12} and PCA (Principal Component Analysis), we offer a construction
which takes a point cloud and generates the geometric information as follows.
}

\begin{itemize}
\item \textbf{Input:} A sequence of points $\xi_{1},...,\xi_{n} \subset
\mathbb{R}^{N}$ and a {bandwidth} parameter $t$.

\bigskip

\item \textbf{Output:} For each point $\xi_{i}$, an approximate basis for a
tangent space at $\xi_{i}$, and Ricci curvature matrix approximating the Ricci
curvature with respect to this basis.

If points are sampled randomly from a smooth submanifold, then for $n$ large
and $t$ small, there will be an orthogonal connection matrix $O_{\xi_{i}%
,\xi_{j}}$ approximating the connection between the tangent spaces of nearby
points $\xi_{i},\xi_{j} $, \cite{SW12}.  
\end{itemize}

There is a choice of kernel and and PCA cutoff parameter, which may affect the
output, but will not affect the limit when the points are sampled from a
smooth submanifold as $n \to\infty.$

Some advantages of our method are the following:

\begin{itemize}
\item It generates a weighted Ricci curvature which takes into account the
underlying probability density.

\item We do not need to approximate the derivatives of the underlying metric.

\item It generates an approximate Ricci curvature even without assuming that
the manifold has a constant dimension

\item Our method allows one to study the convergence of the sample version
of Ricci curvature to its actual value based on extrinsic information like the
\emph{reach} of the submanifold (see Definition \ref{reachdef}).

\item When the connection forms $O_{\xi_{i},\xi_{j}}$ are available, it allows
one to approximate the Hodge Laplacian on 1-forms.
\end{itemize}

\bigskip

The following is our main result: \ For a more precise statement, see Theorem \ref{Ricciapprox} in section \ref{appML}. 

\begin{thmx}
[Approximation of the Ricci Curvature]\label{Main1}Consider the metric measure
space $(\Sigma,\Vert\cdot\Vert,d\mathrm{vol}_{\Sigma})$ where $\Sigma
^{d}\subset\mathbb{R}^{N}$ is a smooth closed embedded submanifold. Suppose
that we have a uniformly distributed i.i.d. sample $\{\xi_{1},\ldots,\xi
_{n}\}$ of points from $\Sigma.$ { For $x\in\Sigma$, one can define a
{sequence of $d$-tuples of orthogonal vectors, and $d\times d$ Ricci matrices
$\hat{R}_{i,j}$ representing the Ricci curvature on these vectors, such at if
$\eta\in T_{x}\Sigma$, then
\[
\left\vert \hat{R}_{i,j}\eta^{i}\eta^{j}-\mathrm{Ric}_{x}(\eta,\eta
)\right\vert \overset{\mathrm{a.s}}{\longrightarrow}0,
\]
}where $\eta^{i}$ are the components of the vector $\eta$ projected onto the
$d$-tuple of vectors approximating the tangent plane. }
\end{thmx}

{We will see that the sequence of tangent spaces can be be constructed using a
method known as \emph{local Principal Component Analysis} (local PCA). Since
this construction is a crucial part of the article \cite{SW12}, we will devote
Section \ref{localPCA} to a fairly detailed explanation of the construction
of vectors using PCA. \ Construction of the Ricci curvature operators is
described in section \ref{background}.}

We will also state, (but not prove) a corresponding theorem that applies when
the sample of points is distributed i.i.d withrespect to a volume form other than the given Riemannian volume form.    \ See Theorem \ref{nonuniform}.

Our method is based on purely on distance, not on combinatorics. In short, we iterate an approximate, kernel-generated Laplacian to obtain a Bakry-Emery operator. To obtain a ``coarse Ricci curvature" this is applied to a function which should be approximately linear at a point. To obtain a Ricci curvature on the PCA basis, we apply this operator to linear functions with differentials determined by the PCA.

Our notion does not assume or attempt to create a triangulation of the manifold.  There have been a number of works recently that study the Forman-Ricci curvature \cite{KKRT,PMT,MR3522779, WeberSJ16}.  Forman-Ricci curvature \cite{Forman} (like ours) is based on the Bakry-Emery approach, but requires a graph structure so is inherently combinatorial.   Indeed, our main results do not contradict results regarding failures of curvature convergence for \textit{triangulated} approximations to a manifold \cite[pg. vii]{LN2184}. 

Our interest in the Ricci curvature is not arbitrary, but is motivated by very
concrete problems in applied mathematics. One example of these problems from topological data analysis : estimating the spectrum of the Hodge Laplacian on 1-forms with
respect to the induced metric of an embedded submanifold $\Sigma$ of
$\mathbb{R}^{N}.$ {Since the seminal paper by Carlsson \cite{GC09} there has been an explosion of interest in topological data analysis, cf.  \cite{GC14}. The spectrum of the Hodge Laplacian} is important in the study of topological properties of
$\Sigma$. Singer and Wu's analysis based on Vector
Diffusion Maps \cite{SW12} represents significant progress in the estimation of the
spectrum of the so called \emph{connection} or \emph{rough} Laplacian on
1-forms. However, there does not yet seem to be an effective way to
estimate the Hodge Laplacian of an embedded manifold. We remark that an
effective algorithm for learning the Ricci curvature of an embedded
submanifold could in principle provide us with a method for estimating the
Hodge Laplacian on 1-forms in view of the \emph{Weitzenb\"{o}ck} formula.
More precisely, given a metric $g$ on $\Sigma$ and a 1-form $X\in\Omega
^{1}(\Sigma^{d})$ we know that
\[
\Delta_{g}X=-\Delta_{H}X+\mathrm{Ric}(X),
\]
where $\mathrm{Ric}(X)$ is the Ricci endomorphism applied to $X$, i.e.,
$\mathrm{Ric}(X)=g^{jk}\mathrm{Ric}_{ij}X_{k}$. {In terms of linear algebra,
if we have taken $n$ points from a $d$ dimensional manifold, with $n$
sufficiently large, the PCA will give us a basis for the tangent space, in
which case the space of 1-forms can be described as a vector in
$\mathbb{R}^{nd}$ - that is, for each point choosing a $d$-vector. The Hodge
Laplacian then becomes a map
\[
\Delta_{H}:\mathbb{R}^{nd}\rightarrow\mathbb{R}^{nd}%
\]
which one can analyze. In the current paper, we do not attempt to prove
spectral convergence, as one can see from the \cite{SW13} that this is
expected to be somewhat involved. }  

Another motivation for the problem of learning the Ricci curvature of a
submanifold is to {measure the robustness of networks}, in particular in biology \cite{Bio1, WSJ2,TSZ+,PMT}. In \cite{TSZ+}, a connection has
been established between Ricci curvature of graphs and the robustness of
cancer networks. Moreover, it has been suggested that robustness of cancer
networks is associated to a certain ``entropy" and that the Ricci curvature of
a graph is closely related to such entropy. The notion of Ricci curvature used
in \cite{TSZ+} is Ollivier's coarse Ricci curvature.  { The recent paper \cite{PMT} compares different notions of Ricci curvature when applied to biological networks. Again we note that most of the methods are based on combinatorial approaches.  }

As we will see, our method is based on the fact that it is possible to
estimate the \emph{rough Laplacian} of the induced metric of an embedded
submanifold of $\mathbb{R}^{N}$. Given an embedded submanifold $\Sigma^{d}$ of
$\mathbb{R}^{N}$, and an embedding $F:\Sigma^{d}\rightarrow\mathbb{R}^{N}$,
the \emph{metric induced by $F$} is given in coordinates by $g_{ij}=\langle
D_{i}F,D_{j}F\rangle$ where $\langle\cdot,\cdot\rangle$ is the Euclidean inner
product in $\mathbb{R}^{N}$ and $D$ means {differentiation with respect to
some coordinates on $\Sigma^{d}$.} By rough Laplacian of $g$ we mean the
operator defined on functions by $\Delta_{g}f=g^{ij}\nabla_{i}\nabla_{j}f$
where $\nabla$ is the Levi-Civita connection of $g$ and $f$ is a
smooth function defined on $\Sigma$. Belkin and Niyogi showed in \cite{BN08}
that given a uniformly distributed point cloud on $\Sigma$ there is a
1-parameter family of operators $L_{t},$ which converge to the
Laplace-Beltrami operator $\Delta_{g}$ on the submanifold. More precisely, the
construction of the operators $L_{t}$ is based on an approximation of the heat
kernel of $\Delta_{g}$, and in particular the bandwidth parameter $t$ can be interpreted
as a choice of scale. In order to learn the rough Laplacian $\Delta_{g}$ from
a point cloud it is necessary to write a sample version of the operators
$L_{t.}$ Then, supposing we have $n$ data points that are independent and
identically distributed (abbreviated by i.i.d.) one can choose a bandwidth parameter $t_{n}$
in such a way that the operators $L_{t_{n}}$ converge almost surely to the
rough Laplacian $\Delta_{g}$. This step follows essentially from applying a
quantitative version of the \emph{law of large numbers}. Thus one can almost
surely learn spectral properties of a manifold. While in \cite{BN08} it is
assumed that the sample is uniform, it was proved by Coifman and Lafon in
\cite{CoLaf06} that if one assumes more generally that the distribution of the
data points has a smooth, strictly positive density in $\Sigma,$ then it is
possible to normalize the operators $L_{t}$ in \cite{BN08} to recover the
rough Laplacian. More generally, the results in \cite{CoLaf06} and \cite{SW12}
show that it is possible to recover a whole family of operators that include
the Fokker-Planck operator and the weighted Laplacian $\Delta_{\rho}f=\Delta
f-\langle\nabla\rho,\nabla f\rangle$ associated to the smooth metric measure
space $(M,g,e^{-\rho}d\mathrm{vol})$, where $\rho$ is a smooth function. \ As
the Bakry-Emery Ricci tensor can be obtained by iterating $\Delta_{\rho}$, the
Bakry-Emery Ricci tensor can be approximated by iterating approximations of
$\Delta_{\rho}.$ \ \ Following \cite{BN08}, Singer and Wu have recently
developed methods for learning the rough Laplacian of an embedded submanifold
on 1-forms using \emph{Vector Diffusion Maps} (VDM) (see for example
\cite{SW12}).

It is the goal of this paper, together with \cite{AW1} and \cite{AW2}, to
demonstrate that the above discussed approximation of the rough Laplacian can
be continued to approximate Ricci curvature as well. In fact, our
approximation method is based on writing sample counterparts of the Ricci
curvature. More generally, we will show that it is possible to show that one
can define sample counterparts of more general objects, for example of the
notions of \emph{Carr\'{e} du Champ} and \emph{iterated Carr\'{e} du Champ}
associated to a diffusion semi-group. Our idea for estimating the Carr\'{e} du
Champ (and ultimately the Ricci curvature) from a sample is closely related to
the results in \cite{AW1,AW2}. For example, in \cite{AW2} we define a family
of coarse Ricci curvatures which depend on a scale parameter $t,$ and show
that when taken on a smooth embedded submanifold on Euclidean space, these
recover the Ricci curvature as $t\rightarrow0$. We will show that as long as
we sample points adequately from the submanifold $\Sigma$, it is possible to
choose a scale $t_{n}$ depending only on the size of the data set (equal to
$n$) to obtain almost sure convergence to the actual Ricci curvature of the
submanifold at a given point. We will summarize the results in \cite{AW1,AW2}
relevant to the present article in Section \ref{summary} (for example Theorem
\ref{extrinsic bias} and Proposition \ref{tconvergence}).

Our results show that one can give a definition of a sample version of Ricci
curvature at a scale on general metric measure spaces that converges to the
actual Ricci curvature on smooth Riemannian manifolds. Moreover, our
definition of empirical coarse Ricci curvature at a scale can be thought of as
an extension of Ricci curvature to a class of discrete metric spaces, namely
those obtained from sampling points from a smooth closed embedded submanifold
of $\mathbb{R}^{N}$. Note however, that in order to obtain convergence of the
empirical coarse Ricci curvature at a scale to the actual Ricci curvature we
need to assume that there is a manifold which fits the distribution of the
data. Recently, Fefferman-Mitter-Narayanan in \cite{FMN13} have developed an
algorithm for testing the hypothesis that there exists a manifold which fits
the distribution of a sample, however, a problem that remains open is how to {best}
estimate the dimension of a submanifold from a sample of points.  {See section \ref{five3} for more discussion.}

In another vein, there is much current interest in a converse problem : The
development of algorithms for generating point clouds on manifolds or even on
surfaces. Recently, there has been progress in this direction by
Karcher-Palais-Palais in \cite{KPP}, specifically on methods for generating
point clouds on implicit surfaces using Monte Carlo simulation and the
Cauchy-Crofton formula.

\subsection{Computational issues}
A naive dense implementation of our algorithm is undoubtably slow, and becomes unfeasible for large numbers of points.  Indeed, inspecting the formulas in section \ref{estimators}, we see that to compute the value of the coarse Ricci curvature at a point $(x,y)$ requires on order of $n^3$ computations. This reduces to $O(n^2)$ if we are in the presence of  a constant density based on assuming also that one knows the dimension of the underlying manifold . Note that $O(n^2)$ is the order of naive dense computation of an optimal transport problem via linear programming.   Thus computing the full coarse Ricci object should require $O(n^5)$. Of course, even storing the distance matrix itself can become prohibitive for large enough $n$.  Nonetheless, for smaller sets (say of size $n<500$) we were able construct a ``coarse Ricci flow'' based on our construction that was successful in clustering some simple data sets \cite{gh2015}. On the other hand, it appears that improvements can be made with very little loss in accuracy, by using only nearest neighbors or truncating the kernel, obtaining a sparse distance function and sparse coarse Ricci object.  

\subsection{Organization of the paper}

This paper is devoted to proving Theorems \ref{Ricciapprox} and
\ref{extrinsic variance}, which are the ingredients for Theorem \ref{Main1}.
\ In the section \ref{background} we review the relevant background. In
Section \ref{summary} we summarize some of the results in \cite{AW2}. The core
of the paper will be Section \ref{empirical} devoted to the prove of Theorem
\ref{extrinsic variance}. In Section \ref{localPCA} we review the construction
of local PCA in \cite{SW12} and show how can we combine this construction with
Theorem \ref{extrinsic variance} to prove Theorem \ref{Ricciapprox}. We also include a discussion of dimension reduction and estimation.

For the reader's convenience we provide a table of notations at the end of section \ref{background}.

\subsection{Acknowledgements}

The authors would like to thank Amit Singer, Hau-Tieng Wu and Charles
Fefferman for constant encouragement. The first author would like to express
gratitude to Adolfo Quiroz for very useful conversations on the topic of
empirical processes, and to Richard Palais for bringing his work to the
attention of both authors. The second author would like to thank Jan Maas for
useful conversations, and Matthew Kahle for stoking his interest in the topic. { We would also like to thank Bartek Siudeja for some code implementations of the algorithm.} 

\section{Background and Definitions}

\label{background}

In this section we recall (cf. \cite{AW1})  how Ricci curvature on general
metric spaces can be constructed with an operator, in particular the
infinitesimal generator of a diffusion semi-group. When the space is a metric
measure space, we use a family of operators which are intended to approximate
a Laplace operator on the space at scale $t.$ As this definition holds on
metric measure spaces constructed from sampling points from a manifold, we can
define an \emph{empirical or sample version} of the Ricci curvature, given a {bandwidth} parameter $t$. As mentioned above, this last construction will have an application
to the manifold learning problem, namely it will serve to predict the Ricci
curvature of an embedded submanifold of $\mathbb{R}^{N}$ if one only has a
point cloud on the manifold and the distribution of the sample has a smooth
positive density. 

\subsubsection{Carr\'{e} du champ}

{We now recall how Bakry and Emery \cite{BE85} related the notion of
\emph{Carr\'{e} du Champ} to Ricci curvature. }Let $P_{t}$ be a $1$-parameter
family of operators of the form
\begin{align}
P_{t}f(x)=\int_{M}f(y)p_{t}(x,dy),
\end{align}
where $f$ is a bounded measurable function defined on $M$ and $p_{t}(x,dy)$ is
a non-negative kernel. We assume that $P_{t}$ satisfies the semi-group
property, i.e.
\begin{align}
P_{t+s}  &  =P_{t}\circ P_{s}.\\
P_{0}  &  =\mathrm{Id}.
\end{align}
In $\mathbb{R}^{n}$, an example of $P_{t}$ is the \emph{Brownian motion},
defined by the density
\begin{align}
p_{t}(x,dy)=\frac{1}{(2\pi t)^{n/2}}e^{-\frac{|x-y|^{2}}{2t}}dy, t\ge0.
\end{align}
If now $P_{t}$ is a diffusion semi-group defined on $(M,g)$, we let $L$ be the
infinitesimal generator of $P_{t},$ which is densely defined in $L^{2}$ by%

\begin{align}
Lf=\lim_{t\rightarrow0}t^{-1}(P_{t}f-f).
\end{align}

We consider a bilinear form which has been introduced in potential theory by
J.P. Roth \cite{Roth74} and by Kunita in probability theory \cite{Kunita69}
and measures the failure of $L$ from satisfying the Leibnitz rule. This
bilinear form is defined as
\begin{equation}
\Gamma(L,u,v)=\frac{1}{2}\left(  L(uv)-L(u)v-uL(v)\right)  . \label{cdc111}%
\end{equation}

When $L$ is the rough Laplacian with respect to the metric $g$, then
\[
\Gamma(\Delta_{g},u,v)=\langle\nabla u,\nabla v\rangle_{g}.
\]
We will also consider the \emph{iterated Carr\'{e} du Champ} introduced by
Bakry and Emery denoted by $\Gamma_{2}$ and defined by
\begin{equation}
\Gamma_{2}(L,u,v)=\frac{1}{2}\left(  L(\Gamma(L,u,v))-\Gamma(L,Lu,v)-\Gamma
(L,u,Lv)\right)  . \label{cdci}%
\end{equation}

Note that if we restrict our attention to the case $L=\Delta_{g}$ the Bochner
formula yields
\begin{align}
\Gamma_{2}(\Delta_{g},u,v)  &  =\frac{1}{2}\Delta\langle\nabla u,\nabla
v\rangle_{g}-\frac{1}{2}\langle\nabla\Delta_{g}u,\nabla v\rangle_{g}-\frac
{1}{2}\langle\nabla u,\nabla\Delta_{g}v\rangle_{g}\nonumber\\
&  =\mathrm{Ric}(\nabla u,\nabla v)+\langle\nabla^{2}_{g}u,\nabla^{2}%
_{g}v\rangle_{g}. \label{bochner}%
\end{align}

We observe immediately that if $\nabla u = e_{i} $ and $\nabla^{2}_{g}u=0$ one
can recover the Ricci tensor via
\begin{align}
\label{recoverRicci}\Gamma_{2}(\Delta_{g},u,u)  &  =\mathrm{Ric}(e_{i},e_{i}).
\end{align}

A geometric interpretation for the Carr\'{e} du Champ in
\eqref{cdc111} and its iterate \eqref{cdci} is given by their role in
formulating the so-called \emph{$CD(K,N)$ curvature condition} due to Bakry
and Emery. The fundamental observation of Bakry and Emery is that the
properties of Ricci curvature lower bounds can be observed and exploited by
using the bilinear form $\Gamma_{2}.$ With this in mind, they define a
curvature-dimension condition for an operator $L$ on a smooth metric measure
space $(X,g,d\nu)$ as follows. If there exist measurable functions
$k:X\rightarrow\mathbb{R}$ and $N:X\rightarrow[1,\infty]$ such that
for every $f$ on a set of functions dense in $L^{2}(X,d\nu)$ the inequality
\begin{align}
\label{cdkn}\Gamma_{2}(L,f,f)\ge\frac{1}{N}(Lf)^{2}+k\Gamma(L,f,f)
\end{align}
holds, then the space $X$ together with the operator $L$ satisfies the
$CD(k,N)$ \emph{condition}, where $k$ stands for curvature and $N$ for
dimension. In \cite{BE85}, it is shown that when considering a smooth metric
measure space $(M^{n},g,e^{-\rho}d\mathrm{vol})$ one has a natural diffusion
operator given by %

\begin{align}
\Delta_{\rho} u=\Delta u-\langle\nabla\rho,\nabla u\rangle,
\end{align}
corresponding to the variation of the Dirichlet energy with respect to the
measure $e^{-\rho}d\mathrm{vol}$. By studying the properties of $\Delta_{\rho
}$, Bakry and Emery arrive at the following dimension and weight dependent
definition of the Ricci tensor:%

\begin{align}
\mathrm{Ric}_{N} & =\left\{
\begin{array}
[c]{ll}%
\mathrm{Ric}+\mathrm{Hess}_{\rho} & \text{if}~N=\infty,\\
\mathrm{Ric}+\mathrm{Hess}_{\rho}-\frac{1}{N-n}(d\rho\otimes d\rho) &
\text{if}~n<N<\infty,\\
\mathrm{Ric}+\mathrm{Hess}_{\rho}-\infty(d\rho\otimes d\rho) & \text{if}%
~N=n,\\
-\infty & \text{if}~N<n,
\end{array}
\right.
\end{align}
and moreover, they showed the equivalence between the $CD(k,N)$ condition
\eqref{cdkn} and the bound $\mathrm{Ric}_{N}\ge k$. The $CD(k,N)$ condition
that we stated above is also related to the \emph{displacement convexity}
condition used by Lott and Villani in \cite{LV09} to study the stability of
lower bounds on Ricci curvature under limits in the Gromov-Hausdorff sense.
For the problem of estimating Ricci curvature from a point cloud, the Bochner
formula \eqref{bochner} and \eqref{recoverRicci} give a direct connection
between the Carr\'{e} du Champ, its iterate and Ricci curvature. 


\subsubsection{Approximations of the Laplacian, Carr\'{e} du Champ and its
iterate.}

Following \cite{BN08} and \cite{CoLaf06}, we recall how to construct operators
which can be thought of as approximations of the Laplacian on metric measure
spaces. Consider a metric measure space $(X,d,\mu)$ with a Borel $\sigma
$-algebra such that $\mu(X)<\infty$. Given $t>0$, let $\theta_{t}$ be given
by
\begin{equation}
\theta_{t}(x)=\int_{X}e^{-\frac{d^{2}(x,y)}{2t}}d\mu(y).
\label{generaldensity}%
\end{equation}
We define a 1-parameter family of operators $L_{t}$ as follows: given a
function $f$ on $X$ let
\begin{equation}
L_{t}f(x)=\frac{2}{t\theta_{t}(x)}\int_{X}\left(  f(y)-f(x)\right)
e^{-\frac{d^{2}(x,y)}{2t}}d\mu(y). \label{Ltdefine}%
\end{equation}
With respect to this $L_{t}$ one can define a Carr\'{e} du Champ on
appropriately integrable functions $f,h$ by
\begin{equation}
\Gamma(L_{t},f,h)=\frac{1}{2}\left(  L_{t}(fh)-(L_{t}f)h-f(L_{t}h)\right)  ,
\label{cdctdef}%
\end{equation}
which simplifies to
\begin{equation}
\Gamma(L_{t},f,h)(x)=\frac{1}{t\theta_{t}(x)}\int_{X}e^{-\frac{d^{2}(x,y)}%
{2t}}(f(y)-f(x))(h(y)-h(x))d\mu(y). \label{cdctsimp}%
\end{equation}
In a similar fashion we define the iterated Carr\'{e} du Champ of $L_{t}$ to
be
\begin{equation}
\Gamma_{2}(L_{t},f,h)=\frac{1}{2}\left(  L_{t}(\Gamma(L_{t},f,h))-\Gamma
(L_{t},L_{t}f,h)-\Gamma(L_{t},f,L_{t}h)\right)  . \label{cdc2t}%
\end{equation}

\begin{remark}\em{
Note that Belking and Niyogi \cite[pg 1295, eq (6)]{BN08} normalize by a
factor $(4\pi t)^{d/2}$, which requires knowledge of the dimension. Our
definition of $L_{t}$ \ (\ref{Ltdefine}) differs from Belkin-Niyogi operator
in that we normalize by $\theta_{t}(x)$ instead. This has a cost in that
convergence may be slower, but has the advantage of being dimensionless,
{allowing our general discussion to fit into the framework of spaces with
lower Ricci curvature bound, for example, the disjoint union of two manifolds
of different dimensions or a sequence of manifolds which may be collapsing. }}
\end{remark}

\subsubsection{Empirical Carr\'{e} du Champ at a given scale}

We can also define empirical versions of $L_{t},\Gamma(L_{t},\cdot,\cdot)$ and
$\Gamma_{2}(L_{t},\cdot,\cdot)$. On a space which consists of $n$ points
$\left\{  \xi_{1},...,\xi_{n}\right\}  $ sampled from a manifold, it is
natural to consider the \emph{empirical measure} defined by
\begin{equation}
\mu_{n}=\frac{1}{n}\sum_{i=1}^{n}\delta_{\xi_{i}} \label{mudefi}%
\end{equation}
where $\delta_{\xi_{i}}$ is the atomic point measure at the point $\xi_{i}$
(also called $\delta$-mass). For any function $f:X\rightarrow\mathbb{R}$ we
will use the notation
\[
\mu_{n}f=\int_{X}f(y)d\mu_{n}(y)=\frac{1}{n}\sum_{j=1}^{n}f(\xi_{j}).
\]

\begin{notation}
\emph{{ We will use the ``hat" notation (for example $\hat{L}_{t}$) to
distinguish those operators, measures, or $t$-densities that have been
constructed from a sample of finite points.} }
\end{notation}

To be more precise, we define the operator $\hat{L}_{t}$ as
\begin{equation}
\hat{L}_{t}f(x)=\frac{2}{t\hat{\theta}_{t}(x)}\int_{X}\left(
f(y)-f(x)\right)  e^{-\frac{d^{2}(x,y)}{2t}}d\mu_{n}(y),
\label{sampledoperator}%
\end{equation}
where
\begin{align}
\hat{\theta}_{t}(x)=\int_{X}e^{-\frac{d^{2}(x,y)}{2t}}d\mu_{n}(y)=\frac{1}%
{n}\sum_{j=1}^{n}e^{-\frac{d(\xi_{j},x)^{2}}{2t}},
\end{align}
and of course
\begin{align}
\int_{X}\left(  f(y)-f(x)\right)  e^{-\frac{d^{2}(x,y)}{2t}}d\mu_{n}%
(y)=\frac{1}{n}\sum_{j=1}^{n}e^{-\frac{d(\xi_{j},x)^{2}}{2t}}(f(\xi
_{j})-f(x)).
\end{align}


The \emph{sample version of Carr\'{e} du Champ} will be the bilinear form
$\Gamma(\hat{L}_{t},f,h)$ which from \eqref{cdctsimp} takes the form
\begin{equation}
\Gamma(\hat{L}_{t},f,h)(x)=\frac{1}{t\hat{\theta}_{t}(x)}\frac{1}{n}\sum
_{j=1}^{n}e^{-\frac{d^{2}(\xi_{j},x)}{2t}}(f(\xi_{j})-f(x))(h(\xi_{j})-h(x)).
\label{empcdc1}%
\end{equation}
We denote the iterated Carr\'{e} du Champ corresponding to $\hat{L}_{t}$ by
$\Gamma_{2}(\hat{L}_{t},f,h)$, and by this we mean
\begin{equation}
\Gamma_{2}(\hat{L}_{t},f,h)=\frac{1}{2}\left(  \hat{L}_{t}(\Gamma(\hat{L}%
_{t}f,h))-\Gamma(\hat{L}_{t},\hat{L}_{t}f,h)-\Gamma(\hat{L}_{t},f,\hat{L}%
_{t}h)\right)  . \label{gamma2general}%
\end{equation}

\subsection{Statement of Results}

\subsubsection{Applications to Manifold Learning}
\label{appML}

We now show how our notion of empirical Carr\'{e} du Champ at a given scale
has applications to the Manifold Learning Problem. For the rest of subsection
\ref{appML} we will consider a closed, smooth, embedded submanifold $\Sigma$
of $\mathbb{R}^{N}$, and the metric measure space will be $(\Sigma
,\|\cdot\|,d\mathrm{vol})$, where

\begin{itemize}
\item $\|\cdot\|$ is the distance function in the ambient space $\mathbb{R}%
^{N}$,

\item $d\mathrm{vol}_{\Sigma}$ is the volume element corresponding to the
metric $g$ induced by the embedding of $\Sigma$ into $\mathbb{R}^{N}$.
\end{itemize}

In addition we will adopt the following conventions

\begin{itemize}
\item All operators $L_{t}$, $\Gamma(L_{t},\cdot,\cdot)$ and $\Gamma_{2}%
(L_{t},\cdot,\cdot)$ will be taken with respect to the distance $\|\cdot\|$
and the measure $d\mathrm{vol}_{\Sigma}$.

\item All sample versions $\hat{L}_{t}$, $\Gamma(\hat{L}_{t},\cdot,\cdot)$ and
$\Gamma_{2}(\hat{L}_{t},\cdot,\cdot)$ are taken with respect to the ambient
distance $\|\cdot\|$.
\end{itemize}

The choice of the above metric measure space is consistent with the setting of
manifold learning in which no assumption on the geometry of the submanifold
$\Sigma$ is made, in particular, we have no a priori knowledge of the geodesic
distance and therefore we can only hope to use the chordal distance as a
reasonable approximation for the geodesic distance. We will show that while
our construction at a scale $t$ involves only information from the ambient
space, the limit as $t$ tends to $0$ will recover the Ricci curvature of the
submanifold. As pointed out by Belkin-Niyogi \cite[Lemma 4.3]{BN08}, the
chordal and intrinsic distance squared functions on a smooth submanifold
disagree first at fourth order near a point , so while much of the analysis is
done on submanifolds, the intrinsic geometry will be recovered in the
limit.

We now address the problem of choosing a bandwidth parameter depending on the size of the
data and the dimension of the submanifold $\Sigma,$ such that the sequence of
empirical Ricci curvatures corresponding to the size of the data converge
almost surely to the actual Ricci curvature of $\Sigma$ at a point. In order
to simplify the presentation of our results, we start by stating the simplest
possible case, which corresponds to a uniformly distributed i.i.d.sample
$\{\xi_{1},\ldots,\xi_{n}\}.$ The more general case of distributions with
strictly positive density with respect to the Lebesgue measure can be explored using the same methods presented here, but the proof is quite lengthy.


\begin{thmx}
[Approximation of the Ricci Curvature]\label{Ricciapprox}Consider the metric
measure space $(\Sigma,\Vert\cdot\Vert,d\mathrm{vol}_{\Sigma})$ where
$\Sigma^{d}\subset\mathbb{R}^{N}$ is a smooth closed embedded submanifold.
Suppose that we have a uniformly distributed i.i.d. sample $\{\xi_{1}%
,\ldots,\xi_{n}\}$ of points from $\Sigma.$ For $\sigma>0,$ let
\begin{equation}
t_{n}=n^{-\frac{1}{3d+3+\sigma}}. \label{choicetns2}%
\end{equation}
{ For $x\in\Sigma$ there exists a {sequence of $d$-tuples of orthogonal
vectors, and $d \times d $ Ricci matrices $\hat{R}_{i,j}$ representing the
Ricci curvature on these vectors, such at if $\eta\in T_{x}\Sigma$, then
\begin{align}
\left\vert \hat{R}_{i,j}\eta^{i}\eta^{j}-\mathrm{Ric}_{x}(\eta,\eta
)\right\vert \overset{\mathrm{a.s}}{\longrightarrow} 0,
\end{align}
}where $\eta^{i}$ are the components of the vector $\eta$ projected onto the
$d$-tuple of vectors approximating the tangent plane. }
\end{thmx}

{See Section \ref{localPCA} for more on the PCA construction. Even though the
approximation method used to obtain Theorem \ref{Ricciapprox} is inspired by
the notion of \emph{Coarse Ricci} curvature introduced by the authors in
\cite{AW1,AW2}, Theorem \ref{Ricciapprox} relies heavily on a precise
estimation of the tangent space at a point by means of local PCA, as opposed
to the approximation proposed in \cite{AW1,AW2} based on the construction of
an \textquotedblleft auxiliary tangent space" by taking segments within the
point cloud. 
\begin{remark}
\emph{{ The convergence is more certain (but perhaps slower)\ if $t_{n}$ is
chosen to go to zero slower than in (\ref{choicetns2}), { i.e.,
slower than
\[
t_{n}=n^{-\frac{1}{3d+3+\sigma+\tau}},
\]
where }}}$\tau$\emph{{{ is any positive number}. In particular, if one
replaces $d$ with an upper bound on $d$, then Theorems \ref{Ricciapprox}, \ref{extrinsic variance},
and Corollary \ref{corunif} still hold.} {
We also remark that we do not attempt compute the rate which one minimizes the mean-squared error of the estimation - that is we do not attempt to compute or justify a version of \textit{Silverman's rule of thumb} (see \cite{silver86}). It could be that there is bandwidth parameter which gives a faster but less sure (i.e. with lower probability) convergence rate.} }
\end{remark}

Besides local PCA, the proof of Theorem \ref{Ricciapprox} relies heavily on
the following theorem for the iterated Carr\'{e} du Champ: Theorem
\ref{Ricciapprox} in turn follows from an approximation result for the
iterated Carr\'{e} du Champ:

\begin{thmx}
\label{extrinsic variance} Consider the metric measure space $(\Sigma
,\Vert\cdot\Vert,d\mathrm{vol}_{\Sigma})$ where $\Sigma^{d}\subset
\mathbb{R}^{N}$ is a smooth closed embedded submanifold. { For $\sigma>0,$
let
\begin{equation}
t_{n}=n^{-\frac{1}{3d+3+\sigma}}.
\end{equation}
}

\begin{enumerate}
\item[(a)] If $f\in C^{\infty}(\mathbb{R}^{N})$, then
\begin{align*}
\sup_{\xi\in\Sigma}\left\vert \hat{\Gamma}_{2}(\hat{L}_{t_{n}},f,f)(\xi
)-\Gamma_{2}(L_{t_{n}},f,f)(\xi)\right\vert \overset{\text{a.s.}%
}{\longrightarrow}0.
\end{align*}

\item[(b)] If $\mathscr{L}_{M}$ is the class of linear functions
$\mathscr{L}_{M}=\{f(\xi)=\langle\zeta,\xi\rangle:\|\zeta\|\le M\}$ where
$\|\cdot\|$ is the ambient distance in $\mathbb{R}^{N}$, then
\begin{align*}
\sup_{f\in\mathscr{L}_{M}}\sup_{\xi\in\Sigma}\left\vert \hat{\Gamma}_{2}%
(\hat{L}_{t_{n}},f,f)(\xi)-\Gamma_{2}(L_{t_{n}},f,f)(\xi)\right\vert
\overset{\text{a.s.}}{\longrightarrow}0.
\end{align*}

\end{enumerate}
\end{thmx}

The proof of Theorem \ref{extrinsic variance} requires using ideas from the
theory of \emph{empirical processes} for which we will provide the necessary
background in Section \ref{empirical}. As pointed out earlier in this
introduction, since we are interested in recovering an object from its sample
version, we are forced to consider a law of large numbers in order to obtain
convergence in probability or almost surely. The problem is that the sample
version of $\Gamma_{2}(L_{t},\cdot,\cdot)$ involves a high correlation between
the data points, destroying independence and any hope of applying large number
results directly. The idea then is to reduce the convergence of the sample
version of $\Gamma_{2}(L_{t},\cdot,\cdot)$ to the application of a uniform law
of large numbers to certain classes of functions. Theorem
\ref{extrinsic variance} is proved in Section \ref{empirical}. In section
\ref{localPCA} we will prove that Theorem \ref{extrinsic variance} indeed
implies Theorem \ref{Ricciapprox}. This will require results from \cite{AW2}.


\subsubsection{Smooth Metric Measure Spaces and non-Uniformly Distributed
Samples}

Consider a smooth metric measure space $\left(  M,g,e^{-\rho}d\mathrm{vol}%
\right)  $ and let $\Delta_{\rho}$ be the operator
\[
\triangle_{\rho}u=\triangle_{g}u-\langle\nabla\rho,\nabla u\rangle_{g}.
\]
In \cite{CoLaf06}, the authors consider a family of operators $L_{t}^{\alpha}$
which converge to $\triangle_{2(1-\alpha)\rho}.$ Note that a standard
computation (cf \cite[Page 384]{vill09}) gives
\[
\Gamma_{2}(\triangle_{2(1-\alpha)\rho},f,f)=\frac{1}{2}\Delta_{g}\left\Vert
\nabla f\right\Vert _{g}^{2}-\langle\nabla\rho,\nabla\Delta_{g}f\rangle
_{g}+2(1-\alpha)\nabla_{g}^{2}\rho(\nabla f,\nabla f).
\]
We adapt \cite{CoLaf06} to our setting: Recall that
\[
\theta_{t}(x)=\int_{X}e^{-\frac{d^{2}(x,y)}{2t}}d\mu(y),
\]
and define, for $\alpha\in\lbrack0,1]$
\begin{equation}
\theta_{t,\alpha}(x)=\int_{X}e^{-\frac{d^{2}(x,y)}{2t}}\frac{1}{\left[
\theta_{t}(y)\right]  ^{\alpha}}d\mu(y). \label{thetaalpha}%
\end{equation}
We can define the operator
\begin{equation}
L_{t}^{\alpha}f(x)=\frac{2}{t}\frac{1}{\theta_{t,\alpha}(x)}\int_{X}
e^{-\frac{d^{2}(x,y)}{2t}}\frac{1}{\left[  \theta_{t}(y)\right]  ^{\alpha}%
}\left(  f(y)-f(x)\right)  d\mu(y), \label{defLtalpha}%
\end{equation}
and again obtain bilinear forms $\Gamma(L_{t}^{\alpha},f,f)$ and $\Gamma
_{2}(L_{\alpha}^{t},f,f)$. For the rest of the section we will consider the
metric measure space $(\Sigma,\Vert\cdot\Vert,e^{-\rho}d\mathrm{vol}_{\Sigma
})$ where $\Sigma^{d}\subset\mathbb{R}^{N}$ is an embedded submanifold,
$\Vert\cdot\Vert$ is the ambient distance and $\rho$ is a smooth function in
$\Sigma$. We again take all the operators $L^{\alpha}_{t},\Gamma_{t}%
(L^{\alpha}_{t},\cdot,\cdot)$ and $\Gamma_{2}(L^{\alpha}_{t},\cdot,\cdot)$ and
their sample counterparts $\hat{L}^{\alpha}_{t},\Gamma_{t}(\hat{L}^{\alpha
}_{t},\cdot,\cdot)$ and $\Gamma_{2}(\hat{L}^{\alpha}_{t},\cdot,\cdot)$ with
respect to the data of $(\Sigma,\Vert\cdot\Vert,e^{-\rho}d\mathrm{vol}%
_{\Sigma})$.

Based on estimates in \cite{AW1} and calculations similar to the proof of
Theorem \ref{extrinsic variance}, we can prove
\begin{thmx}
[Non-uniform case]\label{nonuniform} Consider the metric space $(\Sigma
,\Vert\cdot\Vert)$ where $\Sigma^{d}\subset\mathbb{R}^{N}$ is a smooth closed
embedded submanifold. Suppose that we have an i.i.d. sample $\{\xi_{1}%
,\ldots,\xi_{n}\}$ of points from $\Sigma$ whose common distribution has
density $e^{-\rho}$. For $\sigma>0,$ let
\[
t_{n}=n^{-\frac{1}{4d+4+\sigma}}.
\]
Then, for any $x\in\Sigma$ and any $\eta\in T_{x}\Sigma$ there exists a
sequence of functions $f_{n}$ constructed from the data such that
\[
\left\vert \hat{\Gamma}_{2}^{\alpha}(\hat{L}_{t_{n}}^{\alpha},f_{n},f_{n}%
)(\xi)-\mathrm{Ric}_{x,\alpha}(\eta,\eta)\right\vert \overset{\text{a.s.}%
}{\longrightarrow}0,
\]
and where
\begin{equation}
\mathrm{Ric}_{\alpha}=\mathrm{Ric}+2(1-\alpha)\nabla_{g}^{2}\rho.
\label{wrwpa}%
\end{equation}
\end{thmx}

We omit the proof of Theorem \ref{nonuniform}. \ Heuristically, the order of
decay of $t_{n}$ is not hard to determine, but the proof is exceeding long and tedious.


\subsection{Notation and table of definitions}

Below we give a table of object that frequently appear.\bigskip
\bigskip
\newpage
\captionof{table}{Summary of Notations used in the paper}
\begin{tabular}
[c]{|p{3cm}|p{8cm}|p{3cm}|}\hline
Notation & Description & Where Found\\\hline\hline
$\Sigma^{d}$ & Embedded submanifold of $\mathbb{R}^{N}$ & \\
$N,d$ & Dimensions of ambient space and submanifold & \\
$n$ & Number of sample points & \\
$t$ & Bandwidth parameter & \\
$\xi_{i}$ & point sampled from $\Sigma^{d}$ & \\
$\Gamma(L,u,v)$ & First Carr\'{e} du champ & (\ref{cdc111})\\
$\Gamma_{2}(L,u,v)$ & Iterated Carr\'{e} du champ & (\ref{cdci})\\
$L_{t}$ & Finite scaled approximate Laplacian & (\ref{Ltdefine})\\
$\Delta_{g}$ & Laplace-Beltrami operator on a manifold & \\
$\theta_{t}(x)$ & Scale-$t$ density at $x$ & (\ref{generaldensity})\\
$\mu_{n}$ & measure determined by sampling $n$-points & (\ref{mudefi})\\
$\hat{L}_{t}$ & Empirical Laplace at scale $t$ & (\ref{sampledoperator})\\
$\Gamma(\hat{L}_{t},f,h)$ & Sampled Carr\'{e} du Champ & (\ref{empcdc1})\\
$\Gamma_{2}(\hat{L}_{t},f,h)$ & Sampled iterated Carr\'{e} du Champ &
(\ref{gamma2general})\\
$\Vert\cdot\Vert$ & distance function on $\mathbb{R}^{N}$ & \\
$t_{n}$ & Scale parametric chosen based on $n$ & (\ref{choicetn})\\
$F_{x,y}(z)$ & Approximate signed distance function from $x$ to $y$ &
(\ref{defBIGF})\\
$f_{x,y}(z)$ & Scaled approximate signed distance function & (\ref{deflf})\\
$\mu$ & Probability distribution on $\Sigma^{d}$ & \\
$\mathcal{N}(\mathcal{F},\delta)$ & Covering number of a set of functions, for
radius $\delta$ & (\ref{conumb})\\
$\Vert f\Vert_{\text{A-t-}\mathrm{Lip}}$ & Almost $t$ Lipschitz norm &
(\ref{atlip})\\
$\mathcal{F}_{f,h}^{t},\mathcal{G}^{t},\mathcal{H}_{h}^{t}$ & Families of
functions & (\ref{Ft})(\ref{Gt}), Lemma \ref{classesFtHt}\\
$\tau$ & The reach of $\Sigma$ & Definition \ref{reachdef}\\
$\mathcal{A}(X,\delta)$ & Ambient covering number of a set $X\subset
\mathbb{R}^{N}$ & (\ref{amconu})\\
$U_{V,\tau}(r)$ & Covering bound function & Theorem \ref{coveringbound}\\
$Q_{t}(\mathcal{F},\varepsilon,C,n)$ &  & (\ref{Qdef})\\
$\mathcal{Q}(t,\varepsilon,n)$ &  & (\ref{qbcbeta})\\
$O_{BC}(\beta)$ &  & (\ref{obcbeta})\\
\emph{$C^{\ast}(\Sigma)$} & Constants depending on $\Sigma$ & Convention
\ref{Cconvention}\\\hline
\end{tabular}
\newline

\pagebreak

We also give a table for the different Ricci curvature operators

\bigskip

\captionof{table}{Notation used for different Ricci Curvature operators}
\begin{tabular}
[c]{|p{3cm}|p{8cm}|p{3cm}|}\hline
Notation & Extended Description & Where Found\\\hline\hline
$\mathrm{Ric}_{L}(x,y)$ & Coarse Ricci operator determined by operator $L$ and
evaluated at $x$ and $y.$ This is roughly quadratic in $d(x,y).$ & Definition
\ref{ricl}\\
$\mathrm{RIC}_{L}(x,y)$ & Life-sized Ricci operator determined by operator $L$
and evaluated at $x$ and $y.$ This need not vanish as $x\rightarrow y.$ &
Definition \ref{lricl}\\
$\mathrm{Ric}_{\Delta_{g}}(x,y)$ & Coarse Ricci operator determined by Laplace
operator $\Delta_{g}$ and evaluated at $x$ and $y.$ This is computable when
$\Delta_{g}$ is available. & Definition \ref{ricl}, applied to $\Delta_{g}$\\
{{$\hat{R}_{i,j}$}} & Ricci matrices defined w.r.t PCA\ basis for $T_{x}%
\Sigma^{d}$ & (\ref{pcaricci})\\
$\mathrm{Ric}_{x}(\cdot,\cdot)$ & Classical Ricci 2-form on $T_{x}\Sigma^{d}$
& \\
$\mathrm{Ric}_{\alpha}$ & Weighted Ricci curvature with parameter $\alpha$ and
density $e^{-\rho}$ & (\ref{wrwpa})\\\hline
\end{tabular}

\section{Summary of previous results. Proof of Theorem \ref{Ricciapprox}}

\label{summary}

We now recall the following result proved in \cite{AW2}:

\begin{theorem}
[See \cite{AW2}]
\label{extrinsic bias} Let $\Sigma^{d}\subset\mathbb{R}^{N}$ be a closed
embedded submanifold, let $g$ be the Riemannian metric induced by the
embedding, and let $(\Sigma,\Vert\cdot\Vert,d\mathrm{vol}_{\Sigma})$ be the
metric measure space defined with respect to the ambient distance. { Given any
$f\in\mathscr{C}_{M}\subset C^{5}(\Sigma)$ where $\mathscr{C}_{M}$ is the
class of functions
\begin{align*}
\mathscr{C}_{M}=\left\{  f:\|f\|_{C^{5}(\Sigma)}\le M\right\}  ,
\end{align*}
there exists a constant $C_{1}$ depending on the geometry of $\Sigma$ and the
function $f$ such that
\[
\sup_{x\in\Sigma}\left\vert \Gamma_{2}(\Delta_{g},f,f)(x)-\Gamma_{2}%
(L_{t},f,f)(x)\right\vert <C_{1}(\Sigma,M)t^{1/2}.
\]
}
\end{theorem}

A fundamental step for proving Theorem \ref{extrinsic bias} is the following
proposition shown in \cite{AW2}: \ For simplicity we will assume that
$(\Sigma,d\mathrm{vol}_{\Sigma})$ has unit volume. Recall the definitions
(\ref{generaldensity}), (\ref{Ltdefine}), \eqref{cdctdef}, \eqref{cdctsimp}
and \eqref{cdc2t}.


\begin{proposition}
[See \cite{AW2}]\label{tconvergence} \label{expansion bias} Suppose that
$\Sigma^{d}$ is a closed, embedded, unit volume submanifold of $\mathbb{R}%
^{N}$. Let also $g$ be the metric induced by the embedding of $\Sigma^{d}$
into $\mathbb{R}^{N}$. For any $x$ in $\Sigma$ and for any functions $f,h$ in
$C^{5}(\Sigma)$ we have
\begin{align}
\frac{(2\pi t)^{d/2}}{\theta_{t}(x)}  &  =1+tG_{1}(x)+t^{3/2}R_{1}%
(x),\label{densityt}\\
\Gamma(L_{t},f,h)(x)  &  =\langle\nabla_{\Sigma} f(x),\nabla_{\Sigma}
h(x)\rangle_{g}+t^{1/2}G_{2}(x,J^{2}(f)(x),J^{2}(h)(x))\label{cdct}\\
&  +tG_{3}(x,J^{3}(f)(x),J^{3}(h)(x))+t^{3/2}R_{2}(x,J^{4}(f)(x),J^{4}%
(h)(x)),\\
L_{t}f(x)  &  =\Delta_{g}f(x)+t^{1/2}G_{4}(x,J^{3}(f)(x))+tG_{5}%
(x,J^{4}f(x))+t^{3/2}R_{3}(x,J^{5}f(x)), \label{Ltexpand}%
\end{align}
and
\begin{equation}
\Gamma_{2}(L_{t},f,f)(x)=\Gamma_{2}(\Delta_{g},f,f)(x)+t^{1/2}R_{5}%
(x,J^{5}f(x)) \label{gamma2approx}%
\end{equation}
where each $G_{i}$ is a locally defined function, which is smooth in its
arguments, and $J^{k}(u)$ is a locally defined $k$-jet of the function $u$.
Also, each $R_{i}$ is a locally defined function of $x$ which is uniformly
bounded in terms of its arguments.
\end{proposition}

We will show in Section \ref{localPCA} that Propositions \ref{extrinsic bias}
and \ref{tconvergence} are needed to prove Theorem \ref{Ricciapprox}.


\subsection{Life-Sized Coarse Ricci Curvature}

\label{lifesized} As mentioned above, in \cite{AW1,AW2}, the authors have
formulated a notion of Coarse Ricci Curvature alternative to Ollivier's Coarse
Ricci curvature. The purpose of this section is to formulate the results of
this paper in terms of the notions developed in \cite{AW1,AW2}. In particular,
we show how Ricci curvature can be approximated using test functions different
to the linear functions in Theorem \ref{Ricciapprox}.

Heuristically, the Bochner formula (\ref{bochner}) defines the Ricci curvature
on the gradient of a function $\nabla u$, up to an error term determined by
the Hessian of $u.$ \ In other words, if $u$ is a function with small Hessian
at a point, then the Bochner formula provides a good approximation for the Ricci
curvature in the direction $\nabla u.$ \ For this reason, given a pair of
points $x$ and $y,$ we attempt to construct a ``linear" function that has
gradient ``pointing" from $x$ to $y.$ For any $x,y,x\neq y\in X,$ we define
\begin{equation}
F_{x,y}(z)=\frac{1}{d(x,y)}\frac{1}{2}\left(  d^{2}(x,y)-d^{2}(y,z)+d^{2}%
(z,x)\right)  . \label{defBIGF}%
\end{equation}
Notice that in Euclidean space, this simplifies to
\[
F_{x,y}(z)=\frac{1}{\left\Vert y-x\right\Vert }\left(  x^{2}+\langle z,
y-x\rangle\right) =\frac{1}{\left\Vert y-x\right\Vert }\left(  x^{2}+
z\cdot(y-x)\right) ,
\]
which has gradient a constant unit vector
\[
\nabla F_{x,y}=\frac{(y-x)}{\left\Vert y-x\right\Vert }.
\]
We also define for $x,y\in X,$ the following function%
\begin{equation}
f_{x,y}(z)=\frac{1}{2}\left(  d^{2}(x,y)-d^{2}(y,z)+d^{2}(z,x)\right)  .
\label{deflf}%
\end{equation}
In\ Euclidean space, this has gradient
\[
\nabla f_{x,y}=(y-x).
\]
The function $F$ is chosen so that the gradient will not vanish as the points
$x,y$ approach each other, whereas the function $f$ is chosen so that a
quadratic form on the gradient $\nabla f_{x,y}$ will be the same order as the
distance-squared function $d^{2}(x,y)$ as $x,y$ approach each other. \ Both
have natural interpretations, as we will see below.

{Note that our definitions (\ref{deflf}, \ref{defBIGF}) are designed as an approximation on general metric spaces, when computing coarse Ricci curvature.  When computing Ricci curvature on a known tangent vector in Euclidean space, such a function is unneccessary.   In the combinatorial setting, one can make a special choice which minimizes the contribution of the Hessian in the Bochner formula, see \cite[pg. 659]{KKRT}.
}

\begin{definition}
\label{coarseRicci} \label{ricl}Given an operator $L$ we define the coarse
Ricci curvature for $L$ as
\[
\mathrm{Ric}_{L}(x,y)=\Gamma_{2}(L,f_{x,y},f_{x,y})(x).
\]

\end{definition}

In principle, the functions in \eqref{defBIGF} and \eqref{deflf} serve as a
substitute of the linear functions in Theorem \ref{Ricciapprox} (see also
Section \ref{localPCA}). {Notice that the quantity in (\ref{coarseRicci}) is
the same order as distance squared. To obtain a quantity that does not vanish
near the diagonal, we use (\ref{defBIGF}): }

\begin{definition}
\label{LcoarseRicci} \label{lricl}Given an operator $L$ we define the
life-sized coarse Ricci curvature for $L$ as
\[
\mathrm{RIC}_{L}(x,y)=\Gamma_{2}(L,F_{x,y},F_{x,y})(x).
\]

\end{definition}

From this, one can define notions of \emph{empirical coarse Ricci curvature},
by taking the sample versions of (\ref{coarseRicci}) and (\ref{LcoarseRicci})

Inspired by Theorems \ref{Ricciapprox} and \ref{extrinsic variance}, the
results at the end of Section \ref{empirical} will lead easily to the following.

\begin{corx}
\label{corunif} Let $\Sigma^{d}\subset\mathbb{R}^{N}$ be an embedded
submanifold and consider the metric measure space $(\Sigma,\Vert\cdot
\Vert,d\mathrm{vol}_{\Sigma})$. Suppose that we have an i.i.d. uniformly
distributed sample $\xi_{1},\ldots,\xi_{n}$ drawn from $\Sigma$. Let
\begin{equation}
t_{n}=n^{-\frac{1}{3d+3+\sigma}}, \label{choicetn2}%
\end{equation}
for any $\sigma>0$. Then%

\[
\sup_{x\in\Sigma}\left\vert \hat{\Gamma}_{2}(L_{t_{n}},F_{x,y},F_{x,y}%
)(x)-\mathrm{RIC}_{\Delta_{g}}(x,y)\right\vert \overset{\text{a.s.}%
}{\longrightarrow}0.
\]
In other words, there is a choice of scale depending on the size of the data
and the dimension of the submanifold for which the corresponding empirical
life-sized coarse Ricci curvatures converge almost surely to the life-sized
coarse Ricci curvature.
\end{corx}

The proof is given at the end of Section \ref{empirical}.

Another result proved in \cite{AW2} is

\begin{corollary}
\label{ricdelta} With the hypotheses of Theorem \ref{extrinsic bias} we have
\[
\mathrm{Ric}_{\Delta_{g}}(x,y)=\lim_{t\rightarrow0}\Gamma_{2}(L_{t}%
,f_{x,y},f_{x,y})(x).
\]

\end{corollary}

We note that the relation between the coarse Ricci curvature and the Ricci
curvature is as follows.

\begin{proposition}
[See \cite{AW2}]\label{coarse-to-ricci} Suppose that $M$ is a smooth
Riemannian manifold. Let $V\in T_{x}M$ with $g(V,V)=1.$ \ Then
\[
\mathrm{Ric}(V,V)=\lim_{\lambda\rightarrow0}\mathrm{RIC}_{\triangle_{g}%
}(x,\exp_{x}\left(  \lambda V\right)  ).
\]

\end{proposition}

\section{Empirical Processes and Convergence. Proof of Theorem
\ref{extrinsic variance}}

\label{empirical}

The goal of this section is to prove Theorem \ref{extrinsic variance}. This
will be done using tools from the theory of empirical processes in order to
establish uniform laws of large numbers in a sense that we will explain in
Sections \ref{secGC} through \ref{secproofLD}. For a standard reference in the
theory of empirical processes, see \cite{vdVW96}. See also \cite{SW13} for
further applications of the theory of empirical processes to the recovery of
diffusion operators from a sample.

\subsection{Estimators of the Carr\'{e} Du Champ and the Iterated Carr\'{e} Du
Champ in the uniform case}\label{estimators}

Let us assume that the measure $\mu$ is the volume measure $d\mathrm{vol}%
_{\Sigma}$. Recall that our formal definition of the Carr\'{e} du Champ of
$L_{t}$ with respect to the uniform distribution is given by
\begin{align}
\label{carre}\Gamma(L_{t}, f,h)=\frac{1}{t}\frac{1}{\theta_{t}(x)}\left(
\int_{\Sigma} e^{-\frac{\| x-y\| ^{2}}{2t}}\left(  f(y)-f(x)\right)  \left(
h(y)-h(x)\right)  d\mu(y)\right)  .
\end{align}

It is clear from \eqref{carre} that a sample estimator of the Carr\'{e} Du
Champ at a point $x$ is given by
\begin{equation}
\hat{\Gamma}(L_{t},f,f)(x)=\frac{1}{t}\frac{1}{\hat{\theta}_{t}(x)}\left(
\frac{1}{n}\sum_{j=1}e^{-\frac{\| x-\xi_{j}\|^{2}}{2t}}\left(  f(\xi
_{j})-f(x)\right)  \left(  h(\xi_{j})-h(x)\right)  \right)  , \label{48}%
\end{equation}
and recall that we defined the $t$-Laplace operator by
\begin{align}
L_{t}f(x)=\frac{2}{t}\frac{1}{\theta_{t}(x)}\int e^{-\frac{\left\|
x-y\right\|  ^{2}}{2t}}\left(  f(y)-f(x)\right)  d\mu(y), \label{49}%
\end{align}
and its sample version is
\begin{equation}
\hat{L}_{t}f(x)=\frac{2}{t}\frac{1}{\hat{\theta}_{t}(x)}\frac{1}{n}\sum
_{j=1}^{n}e^{-\frac{\| x-\xi_{j}\|^{2}}{2t}}\left(  f(\xi_{j})-f(x)\right)  .
\label{50}%
\end{equation}
Recall that the iterated Carr\'{e} du Champ is
\begin{align}
\Gamma_{2}(L_{t},f,h)=\frac{1}{2}\left(  L_{t}\Gamma(L_{t},f,h)-\Gamma
(L_{t},L_{t}f,f)-\Gamma(L_{t},f,L_{t}h)\right)  .
\end{align}
For simplicity, we will evaluate $\Gamma_{2}(L_{t},\cdot,\cdot)$ at a pair
$(f,f)$ instead of $(f,h)$ and by symmetry it is clear that we obtain
\begin{align}
\Gamma_{2}(L_{t},f,f)=\frac{1}{2}\left(  L_{t}(\Gamma_{t}(f,f))-2\Gamma
_{t}(L_{t}f,f)\right)  .
\end{align}
Combining the sample versions of $\Gamma_{t}$ and $L_{t}$ we obtain a sample
version for $\Gamma_{2}(L_{t},f,f)$%

\begin{align}
&  \hat\Gamma_{2}(L_{t},f,f)(x)=\frac{1}{t^{2}}\frac{1}{n^{2}}\sum_{j=1}%
\sum_{k=1}^{n}\frac{1}{\hat{\theta}_{t}(\xi_{k})\hat{\theta}_{t}(x)}%
e^{-\frac{\|x-\xi_{j}\|^{2}}{2t}-\frac{\|\xi_{j}-\xi_{k}\|^{2}}{2t}}(f(\xi
_{j})-f(\xi_{k}))^{2}\label{fullsample1}\\
&  -\frac{1}{t^{2}n^{2}}\sum_{j=1}^{n}\sum_{k=1}^{n}\frac{1}{\hat{\theta}%
^{2}_{t}(x)}e^{-\frac{\|x-\xi_{j}\|^{2}}{2t}-\frac{\|x-\xi_{k}\|^{2}}{2t}%
}(f(\xi_{k})-f(x))^{2}\\
&  -\frac{2}{t^{2}n^{2}}\sum_{j=1}^{n}\sum_{k=1}^{n}\frac{1}{\hat{\theta}%
_{t}(x)\hat{\theta}_{t}(\xi_{j})}e^{-\frac{\|x-\xi_{j}\|^{2}}{2t}-\frac
{\|\xi_{j}-\xi_{k}\|^{2}}{2t}}(f(\xi_{k})-f(\xi_{j}))(f(\xi_{j})-f(x))\\
&  +\frac{2}{t^{2}n^{2}}\sum_{j=1}^{n}\sum_{k=1}^{n}\frac{1}{\hat{\theta}%
^{2}_{t}(x)}e^{-\frac{\|x-\xi_{j}\|^{2}}{2t}-\frac{\|\xi_{j}-\xi_{k}\|^{2}%
}{2t}}(f(\xi_{k})-f(x))(f(\xi_{j})-f(x)). \label{fullsample2}%
\end{align}

In principle, the convergence analysis for
\eqref{fullsample1}-\eqref{fullsample2} can be done using the following
standard result in large deviation theory.

\begin{lemma}
[Hoeffding's Lemma]\label{hoeffding}Let $\xi_{1},\ldots,\xi_{n}$ be i.i.d.
random variables on the probability space $(\Sigma,\mathcal{B},\mu)$ where
$\mathcal{B}$ is the Borel $\sigma$-algebra of $\Sigma$, and let
$f:\Sigma\rightarrow\lbrack-K,K]$ be a Borel measurable function with $K>0$.
Then for the corresponding empirical measure $\mu_{n}$ and any $\varepsilon>0$
we have
\[
\Pr\left\{  |\mu_{n}f-\mu f|\geq\varepsilon\right\}  \leq2e^{-\frac
{\varepsilon^{2}n}{2K^{2}}}.
\]

\end{lemma}

Observe, however, that \eqref{fullsample1}-\eqref{fullsample2} is a non-linear
expression which will involve non-trivial interactions between the data points
$\xi_{1},\ldots,\xi_{n}$. This non-trivial interaction between the points
$\xi_{1},\ldots,\xi_{n}$ will produce a loss of independence and we will
\emph{not} be able to apply Hoeffding's Lemma directly to
\eqref{fullsample1}-\eqref{fullsample2}. In order to address this difficulty
we will establish several uniform laws of large numbers which will provide us
with a large deviation estimate for \eqref{fullsample1}-\eqref{fullsample2}.

\begin{remark}
\label{schemCDC2} We will not use directly the expression
\eqref{fullsample1}-\eqref{fullsample2}, instead we will write
\eqref{fullsample1}-\eqref{fullsample2} schematically in the form
\begin{align}
\hat\Gamma_{2}(L_{t}(f,f)(x)=\frac{1}{2}\left(  \hat{L}_{t}\left(  \hat
{\Gamma}(f,f)\right)  (x)-2\hat{\Gamma}_{t}(\hat{L}_{t}f,f)(x)\right)  ,
\end{align}
which is clearly equivalent to \eqref{fullsample1}-\eqref{fullsample2}.
\end{remark}


\subsection{Glivenko-Cantelli Classes}

\label{secGC} A \emph{Glivenko-Cantelli} class of functions is essentially a
class of functions for which a uniform law of large numbers is satisfied.

\begin{definition}
Let $\mu$ be a fixed probability distribution defined on $\Sigma$. A class
$\mathcal{F}$ of functions of the form $f:\Sigma\rightarrow\mathbb{R}$ is
\emph{Glivenko-Cantelli} if

\begin{enumerate}
\item[(a)] $f\in L^{1}(d\mu)$ for any $f\in\mathcal{F}$,

\item[(b)] For any i.i.d. sample $\xi_{1},\ldots,\xi_{n}$ drawn from $\Sigma$
whose distribution is $\mu$ we have uniform convergence in probability in the
sense that for any $\varepsilon>0$
\begin{align}
\lim_{n\rightarrow\infty}{\Pr}^{*}\left\{  \sup_{f\in\mathcal{F}}|\mu_{n}f-\mu
f|>\varepsilon\right\}  =0.
\end{align}

\end{enumerate}
\end{definition}

\begin{remark}
\emph{{ Note that in general we have to consider outer probabilities ${\Pr
}^{*}$ instead of $\Pr$ because the class $\mathcal{F}$ may not be countable
and the supremum $\sup_{f\in\mathcal{F}}|\mu_{n}f-\mu f|$ may not be
measurable. On the other hand, if the class $\mathcal{F}$ is separable in
$L^{\infty}(\Sigma)$, then we can replace ${\Pr}^{*}$ by $\Pr$. While all of
the classes we will encounter in this paper will be separable in $L^{\infty
}(\Sigma),$ we use ${\Pr}^{*}$ when we deal with a general class.} }
\end{remark}

Let $\mathcal{F}$ be a class of functions defined on $\Sigma$ and totally
bounded in $L^{\infty}(\Sigma)$. Given $\delta>0$ we let $\mathcal{N}%
(\mathcal{F},\delta)$ be the $L^{\infty}$ $\delta$-covering number of
$\mathcal{F}$ , i.e.,
\begin{equation}
\mathcal{N}(\mathcal{F},\delta)=\inf\{m:\mathcal{F}~\text{is covered
by}~m~\text{balls of radius}~\delta~\text{in the}~L^{\infty}~\text{norm}\}.
\label{conumb}%
\end{equation}

\begin{lemma}
\label{equicontinuity} Let $\mathcal{F}$ be an equicontinuous class of
functions in $L^{\infty}(\Sigma)$ that satisfies $\displaystyle{\sup
_{f\in\mathcal{F}}\{\Vert f\Vert_{L^{\infty}(\Sigma)}\}}\leq M<\infty$ for
some $M>0$. Then for any distribution $\mu$ which is absolutely continuous
with respect to $d\mathrm{vol}_{\Sigma}$, the class $\mathcal{F}$ is $\mu
$-Glivenko-Cantelli. Moreover, if $\xi_{1},\ldots,\xi_{n}$ is an i.i.d. sample
drawn from $\Sigma$ with distribution $\mu$ we have
\[
{\Pr}^{\ast}\left\{  \sup_{f\in\mathcal{F}}\left\vert \mu_{n}f-\mu
f\right\vert \geq\varepsilon\right\}  \leq2\mathcal{N}\left(  \mathcal{F}%
,\frac{\varepsilon}{4}\right)  e^{-\frac{\varepsilon^{2}n}{8M^{2}}}.
\]

\end{lemma}

\begin{proof}
By equicontinuity of $\mathcal{F}$, it follows from the Arzel\`{a}-Ascoli
theorem that $\mathcal{F}$ is precompact in the $L^{\infty}(\Sigma)$ norm and
hence totally bounded in $L^{\infty}(\Sigma)$. In particular for every
$\delta>0$, the number $\mathcal{N}(\mathcal{F},\delta)$ is finite. Let
$\mathcal{G}$ be a finite class such that the union of all balls with center
in $\mathcal{G}$ and radius $\delta$ covers $\mathcal{F}$ and $|\mathcal{G}%
|=\mathcal{N}(\mathcal{F},\delta)$. For any $f\in\mathcal{F}$ there exists
$\phi\in\mathcal{G}$ such that $\|f-\phi\|_{L^{\infty}(\Sigma)}<\delta$ and we
obtain
\begin{align}
|\mu_{n}f-\mu f|\le2\delta+|\mu_{n}\phi-\mu\phi|,
\end{align}
and clearly
\begin{align}
\sup_{f\in\mathcal{F}}|\mu_{n}f-\mu f|\le2\delta+\max_{\phi\in\mathcal{G}}%
|\mu_{n}\phi-\mu\phi|.
\end{align}
Fixing $\varepsilon>0$ and choosing $\delta=\varepsilon/4$ we observe that
\begin{align}
{\Pr}^{*}\left\{  \sup_{f\in\mathcal{F}}|\mu_{n}f-\mu f|\ge\varepsilon
\right\}  \le\Pr\left\{  \max_{\phi\in\mathcal{G}}|\mu_{n}\phi-\mu\phi
|\ge\frac{\varepsilon}{ 2}\right\}  ,
\end{align}
and by Hoeffding's inequality we have
\begin{align}
{\Pr}\left\{  \max_{\phi\in\mathcal{G}}|\mu_{n}\phi-\mu\phi|\ge\frac
{\varepsilon}{2}\right\}  \le2\mathcal{N}\left(  \mathcal{F},\frac
{\varepsilon}{4}\right)  e^{-\frac{\varepsilon^{2}n}{8M^{2}}},
\end{align}
which implies the lemma.
\end{proof}

{ Recall that as in the statement of Theorem \ref{Ricciapprox},
we have defined the space $\mathrm{Lip}(\mathbb{R}^{N})$ of functions with
bounded Lipschitz semi-norm in $\mathbb{R}^{N}$. To be clear, we will be using
the following norms and semi-norms:
}%

\begin{equation}
\Vert f\Vert_{\mathrm{Lip}}=\sup_{x,y\in\mathbb{R}^{N},x\neq y}\left\{
\frac{|f(x)-f(y)|}{\Vert x-y\Vert}\right\}  .
\end{equation}

\begin{equation}
\Vert f\Vert_{C^{k}}=\Vert f\Vert_{L^{\infty}(\Sigma)}+\sum_{j=1}^{k}\Vert
D_{j}f\Vert_{L^{\infty}(\Sigma)}. \label{cknorm}%
\end{equation}
where
\[
\Vert D_{j}f\Vert_{L^{\infty}(\Sigma)}=\sup_{x\in\Sigma}\Vert D_{j}f(x)\Vert,
\]
i.e., the norm $\Vert D_{j}f(x)\Vert$ is the norm in the ambient space
$\mathbb{R}^{N}$ , and%

\begin{equation}
\Vert f\Vert_{\text{A-t-}\mathrm{Lip}}=\inf\left\{  A+B:|f(x)-f(y)|\leq A\Vert
x-y\Vert+Bt^{1/2}\text{ for all }x,y\in\mathbb{R}^{N}\right\}  \label{atlip}%
\end{equation}
this $t$-almost Lipschitz norm being weaker than the Lipschitz norm. \
We will frequently use the following classes of functions
\begin{align}
\mathcal{F}_{f,h}^{t}  &  =\left\{  \phi_{t}(\xi,\zeta)=t^{-1/2}%
e^{-\frac{\Vert\xi-\zeta\Vert^{2}}{2t}}(f(\xi)-f(\zeta))(h(\xi)-h(\zeta
)):\xi\in\Sigma\right\}  ,\label{Ft}\\
\mathcal{G}^{t}  &  =\left\{  \psi_{t}(\xi,\zeta)=t^{1/2}e^{-\frac{\Vert
\xi-\zeta\Vert^{2}}{2t}}:\xi\in\Sigma\right\}  , \label{Gt}%
\end{align}
where $f,h$ in \eqref{Ft} are fixed functions. Given a class of functions
$\mathscr{S}$ we use $\mathscr{M}_{\mathscr{S}}$ to denote
\[
\mathscr{M}_{\mathscr{S}}=\sup_{f\in\mathscr{S}}\{\Vert f\Vert_{L^{\infty
}(\Sigma)}\}.
\]

With this notation we easily find that%

\begin{align}
\mathscr{M}_{\mathcal{F}^{t}_{f,h}}  &  =\sup_{\phi\in\mathcal{F}^{t}}%
\{\Vert\phi\Vert_{L^{\infty}(\Sigma)}\}\leq\left(  \frac{2}{e}\right)
t^{1/2}\|f\|_{\mathrm{Lip}}\|h\|_{\mathrm{Lip}},\label{boundFt}\\
\mathscr{M}_{\mathcal{F}_{f,h}^{t}}  &  \leq t^{1/2}\Vert f\Vert
_{\text{A-t-}\mathrm{Lip}}\Vert h\Vert_{\mathrm{Lip}},\\
\mathscr{M}_{\mathcal{G}^{t}}  &  =\sup_{\psi\in\mathcal{G}^{t}}\{\Vert
\psi\Vert_{L^{\infty}(\Sigma)}\}=t^{1/2}. \label{boundGt}%
\end{align}

\subsection{Ambient Covering Numbers}

\label{ambient} In this subsection we show that the computation of covering
numbers of the classes of functions $\mathcal{F}_{f,h}^{t}$ and $\mathcal{G}%
^{t}$ introduced in \eqref{Ft}, \eqref{Gt}
reduces to the computation of covering numbers of submanifolds of
$\mathbb{R}^{N}$. For this purpose we will need the notion of \emph{ambient
covering number of} a totally bounded set $X$ of $\mathbb{R}^{N}$ defined by
\begin{equation}
\mathcal{A}(X,\delta)=\inf\{m:\exists A~\text{with}~|A|=m~\text{and}%
~X\subset\bigcup_{a\in A}B_{\mathbb{R}^{N},\delta}(a)\}. \label{amconu}%
\end{equation}
where the ball $B_{\mathbb{R}^{N},\delta}(a)$ is taken with respect to the
ambient distance $\Vert\cdot\Vert$.

\begin{lemma}
\label{coveringpsi} For any $\xi,\tilde{\xi},\zeta\in\mathbb{R}^{N}$ we have
\begin{align}
\left|  t^{1/2}e^{-\frac{\|\zeta-\xi\|^{2}}{2t}}-t^{1/2}e^{-\frac
{\|\zeta-\tilde{\xi}\|^{2}}{2t}}\right|  \le e^{-1/2}\|\xi-\tilde{\xi}\|.
\end{align}
In particular
\begin{align}
\mathcal{N}(\mathcal{G}^{t},\varepsilon)\le\mathcal{A}(\Sigma,e^{1/2}%
\varepsilon).
\end{align}

\end{lemma}

\begin{proof}
Observe that the function $\psi_{t}(\xi,\zeta)=t^{1/2}e^{-\frac{\|\zeta
-\xi\|^{2}}{2t}}$ satisfies
\begin{align}
D_{\xi}\psi_{t}(\xi,\zeta)=t^{1/2}\frac{(\zeta-\xi)}{t}e^{-\frac{\| \xi
-\zeta\|^{2}}{2t}},
\end{align}
and therefore
\begin{align}
\sup_{\xi,\zeta\in\mathbb{R}^{N}}\left\|  D_{\xi}\psi_{t}(\xi,\zeta)\right\|
\leq\sqrt{2}\sup_{\rho>0}\left\{  \rho e^{-\rho^{2}}\right\}  =e^{-\frac{1}%
{2}},
\end{align}
and therefore we have the Lipschitz estimate
\begin{align}
|\psi_{t}(\xi,\zeta)-\psi_{t}(\tilde{\xi},\zeta)|\leq e^{-1/2}\|\xi-\tilde
{\xi}\|. \label{bracketpsi}%
\end{align}

\end{proof}

\begin{corollary}
\label{linfinityclass}Fix a function $0\ne h\in L^{\infty}(\Sigma)$ with
$\|h\|_{L^{\infty}(\Sigma)}\le C$ and consider the class of functions%

\begin{align}
\label{Ht}\mathcal{H}^{t}_{h}=\left\{  \psi_{t}(\zeta,\cdot)h(\cdot):\zeta
\in\Sigma\right\}  .
\end{align}
Then for every $\varepsilon>0$ we have
\begin{align}
\mathcal{N}(\mathcal{H}^{t}_{h},\varepsilon)\le\mathcal{A}\left(  \Sigma
,\frac{e^{1/2}}{C}\varepsilon\right)  .
\end{align}

\end{corollary}

\begin{proof}
We use Lemma \ref{coveringpsi} to obtain the estimate
\begin{align}
\left\|  \psi_{t}(\zeta,\cdot)h(\cdot)-\psi_{t}(\zeta^{^{\prime}}%
,\cdot)h(\cdot)\right\|  _{L^{\infty}(\Sigma)}  &  \le C\|\psi_{t}(\zeta
,\cdot)-\psi_{t}(\zeta^{\prime},\cdot)\|_{L^{\infty}(\Sigma)}\\
&  \le Ce^{-1/2}\|\zeta-\zeta^{\prime}\|,
\end{align}
from which the corollary follows.
\end{proof}

\begin{lemma}
\label{coveringphi} For any $\phi_{t}(\xi,\cdot),\phi_{t}(\xi^{^{\prime}%
},\cdot)\in\mathcal{F}_{f,h}^{t}$ we have
\[
\sup_{\zeta\in\mathbb{R}^{d}}|\phi_{t}(\xi,\zeta)-\phi_{t}(\xi^{^{\prime}%
},\zeta)|\leq C(f,h)\Vert\xi-\xi^{^{\prime}}\Vert,
\]
where
\begin{align}
C(f,h)=C_{0}\|f\|_{\mathrm{Lip}}\|h\|_{\mathrm{Lip}}, \label{C(f,h)}%
\end{align}
and $C_{0}$ is a universal constant. Thus
\end{lemma}

\begin{align}
\mathcal{N}(\mathcal{F}_{f,h}^{t},\delta)\le\mathcal{A}\left(  \Sigma
,\frac{\delta}{C(f,h)}\right)  .
\end{align}

\begin{proof}
Let $\phi_{t}(\xi,\cdot)\in\mathcal{F}_{f,h}^{t}$. \ Fixing $\zeta$ and
differentiating in $\xi$ we have
\begin{align}
D_{\xi}\phi_{t}(\xi,\zeta)=  &  t^{-1/2}\frac{(\zeta-\xi)}{2t}e^{-\frac
{\Vert\xi-\zeta\Vert^{2}}{2t}}(f(\xi)-f(\zeta))(h(\xi)-h(\zeta))\\
&  +t^{-1/2}e^{-\frac{\Vert\xi-\zeta\Vert^{2}}{2t}}D_{\xi}f(\xi)(h(\xi
)-f(\zeta))\\
&  +t^{-1/2}e^{-\frac{\Vert\xi-\zeta\Vert^{2}}{2t}}D_{\xi}h(\xi)(f(\xi
)-f(\zeta)),
\end{align}
and then
\begin{align}
\Vert D_{\xi}\phi_{t}(\xi,\zeta)\Vert &  \leq\Vert f\Vert_{\mathrm{Lip}}\Vert
h\Vert_{\mathrm{Lip}}\frac{\Vert\xi-\zeta\Vert^{3}}{t^{3/2}}e^{-\frac{\Vert
\xi-\zeta\Vert^{2}}{2t}}\\
&  +\frac{\Vert\xi-\zeta\Vert}{t^{1/2}}e^{-\frac{\Vert\xi-\zeta\Vert^{2}}{2t}%
}\Vert f\Vert_{\mathrm{Lip}}\Vert h\Vert_{\mathrm{Lip}}\\
&  +\frac{\Vert\xi-\zeta\Vert}{t^{1/2}}e^{-\frac{\Vert\xi-\zeta\Vert^{2}}{2t}%
}\Vert h\Vert_{\mathrm{Lip}}\Vert f\Vert_{\mathrm{Lip}}\\
&  \leq C_{0} \Vert f\Vert_{\mathrm{Lip}}\Vert h\Vert_{\mathrm{Lip}},
\end{align}
where
\[
C_{0}=3\max\left(  \sup_{\rho>0}\{\rho^{3}e^{-\rho^{2}/2}\},\sup_{\rho
>0}\{\rho e^{-\rho/2}\}\right)  .
\]
It follows that for any $\zeta\in\Sigma$ and any $\xi,\xi^{^{\prime}}\in
\Sigma$ we have the Lipschitz estimate
\[
|\phi_{t}(\xi,\zeta)-\phi_{t}(\xi^{\prime},\zeta)|\leq C(f,h)\Vert\xi
-\xi^{^{\prime}}\Vert.
\]

\end{proof}

The following can be obtained in a similar fashion.

\begin{lemma}
\label{lipschitzupsilon} Let $f\in C^{1}$ and let $\upsilon_{t}$ be given by
\begin{align*}
\upsilon_{t}(\xi,\zeta)=e^{-\frac{\|\xi-\zeta\|^{2}}{2t}}(f(\xi)-f(\zeta)).
\end{align*}
We then have the following estimate
\begin{align}
\left|  \upsilon_{t}(\xi,\zeta)-\upsilon_{t}(\xi^{^{\prime}},\zeta)\right|
\le\left(  \frac{2t}{e}+1\right)  \|f\|_{\mathrm{Lip}} \|\xi-\xi^{^{\prime}%
}\|.
\end{align}

\begin{proof}
As before,
\begin{align}
|D_{\xi}\upsilon_{t}(\xi,\zeta)|  &  \le\frac{\|\xi-\zeta\|}{t}e^{-\frac
{\|\xi-\zeta\|^{2}}{2t}}|f(\xi)-f(\zeta)|+e^{-\frac{\|\xi-\zeta\|^{2}}{2t}%
}\|D_{\xi}f(\xi)\|\\
&  \le\frac{\|\xi-\zeta\|^{2}}{t}e^{-\frac{\|\xi-\zeta\|^{2}}{2t}%
}\|f\|_{\mathrm{Lip}}+\sup_{x\in\Sigma}\|D_{\xi}f(\xi)\|,
\end{align}
and
\begin{align}
\sup_{\rho>0}\rho^{2}e^{-\rho^{2}}=\frac{2}{e}.
\end{align}

\end{proof}
\end{lemma}

Lemmas \ref{coveringpsi} and \ref{coveringphi} say that we can relate covering
numbers of the classes $\mathcal{F}^{t}_{f,h},\mathcal{G}^{t}$ to ambient
covering numbers of the submanifold $\Sigma$. In order to estimate
$\mathcal{A}(\Sigma,\delta)$ we need to introduce the notion of \emph{reach}
of an embedded submanifold of $\mathbb{R}^{N}$. Recall that for every
$\varepsilon>0$ we can consider the $\varepsilon$ neighborhood of $\Sigma$
\begin{align}
\Sigma_{\varepsilon}=\{x\in\mathbb{R}^{N}:d(x,\Sigma)<\varepsilon\},
\end{align}
where $d(\cdot,\Sigma)$ measures the distance from points in $\mathbb{R}^{N}$
to $\Sigma$ with respect to the ambient norm $\|\cdot\|$. If $\Sigma$ is a
smooth, embedded submanifold of $\mathbb{R}^{N}$, for $\varepsilon>0$
sufficiently small we can define a smooth map $\varphi:\Sigma_{\varepsilon
}\rightarrow\Sigma$ such that

\begin{enumerate}
\item $\varphi$ is smooth,

\item $\varphi(x)$ is the unique point in $\Sigma$ such that $\|\varphi
(x)-x\|=d(x,\Sigma)$ for all $x\in\Sigma_{\varepsilon}$.

\item $x-\varphi(x)\in T^{\perp}_{\varphi(x)}\Sigma$,

\item $\varphi(y+z)\equiv\varphi(y)$ for all $y\in\Sigma$ and $z\in
(T_{y}\Sigma)^{\perp}$ with $\|z\|<\varepsilon$,

\item For any vector $V\in\mathbb{R}^{N}$, $D_{V}\varphi(x)=p_{\varphi(x)}%
(V)$, where $p_{\varphi(x)}(V)$ is the orthogonal projection of $V$ onto
$T_{\varphi(x)}\Sigma$.
\end{enumerate}

See for example \cite[Theorem1]{lsbook}. The map $\varphi$ is called
\emph{nearest point projection} onto $\Sigma$.

\begin{definition}
\label{reachdef} Let $\Sigma$ be an embedded submanifold of $\mathbb{R}^{N}$.
The reach of $\Sigma$ is the number
\begin{align}
\tau=\sup\{{\varepsilon>0} :\text{There exists a nearest point projection
in}~\Sigma_{\varepsilon}\}.
\end{align}

\end{definition}

We quote the following result

\begin{theorem}
[\cite{FMN13}~{Corollary 6}]\label{coveringbound} Suppose that $\Sigma$ is a
$d$-dimensional embedded submanifold of $\mathbb{R}^{N}$ with volume $V$ and
reach $\tau>0$ and let $U:\mathbb{R}_{+}\rightarrow\mathbb{R}_{+}$ be the
function
\begin{align}
U_{V,\tau}(r)=\left(  \frac{1}{\tau^{d}}+r^{d}\right)  V,\label{defU}%
\end{align}
then for any $\varepsilon>0$ there is an $\varepsilon$-net of $\Sigma$ with
respect to the ambient distance $\Vert\cdot\Vert$ of no more than
$C_{d}U(\varepsilon^{-1})$ points where $C_{d}$ is a dimensional constant. In
particular,
\[
\mathcal{A}(\Sigma,\varepsilon)\leq C_{d}U(\varepsilon^{-1}).
\]

\end{theorem}

\begin{corollary}
\label{coveringFG} We have the following bounds

\begin{enumerate}
\item[(a)]
\begin{align}
\mathcal{N}(\mathcal{F}^{t}_{f,h},\varepsilon)\le C_{d}V\left(  \frac{1}%
{\tau^{d}}+\left(  \frac{C(f,h)}{\varepsilon}\right)  ^{d}\right)  ,
\label{N(f,t)}%
\end{align}

\item[(b)]
\begin{align}
\mathcal{N}(\mathcal{G}^{t},\varepsilon)\le C_{d}V\left(  \frac{1}{\tau^{d}%
}+\left(  \frac{1}{e^{1/2}\varepsilon}\right)  ^{d}\right)  , \label{N(G,t)}%
\end{align}

\item[(c)]
\begin{align}
\mathcal{N}(\mathcal{H}^{t}_{h},\varepsilon)\le C_{d}V\left(  \frac{1}%
{\tau^{d}}+\left(  \frac{C}{e^{1/2}\varepsilon}\right)  ^{d}\right)  ,
\label{N(h,t)}%
\end{align}
where $C(f,h)$ in \eqref{N(f,t)} is given by \eqref{C(f,h)} and $C$ in
\eqref{N(h,t)} is an upper bound for $\|h\|_{L^{\infty}(\Sigma)}$.
\end{enumerate}
\end{corollary}

{ }

\begin{remark}
\emph{{ From \eqref{N(f,t)} we can obtain effective bounds for $\mathcal{N}%
(\mathcal{F}^{t}_{f,h},\epsilon)$ in the sense that these bounds depend only
on bounds for $\|f\|_{\mathrm{Lip}},\|h\|_{\mathrm{Lip}}$. } }
\end{remark}

We now prove uniform estimates on the covering number for the collection of
classes $\mathcal{F}^{t}_{f,h}$ and $\mathcal{H}^{t}_{h}$ for $f,h$ in the
class $\mathscr{L}_{M}$ defined in the statement of Theorem
\ref{extrinsic variance}.

{ }

\begin{lemma}
\label{classesFtHt}
Consider the classes of functions
\begin{align}
\mathcal{F}^{t}  &  =\left\{  \phi_{t}(f,h,\xi)(\zeta)=t^{-1/2} e^{-\frac
{\Vert\xi-\zeta\Vert^{2}}{2t}}(f(\xi)-f(\zeta))(h(\xi)-h(\zeta)):\xi\in
\Sigma,f,h\in\mathscr{L}_{M}\right\}  ,\label{defFt}\\
\mathcal{H}^{t}  &  =\left\{  \omega_{t}(h,\xi)(\zeta)= t^{1/2}e^{-\frac
{\|\xi-\zeta\|^{2}}{2t}}h(\zeta):\xi\in\Sigma,h\in\mathscr{L}_{M}\right\}
,\label{defHt}\\
\mathcal{F}_{*}^{t}  &  =\left\{  \phi^{*}_{t}(f,h,\xi)(\zeta)=t^{-1/2}%
e^{-\frac{\Vert\xi-\zeta\Vert^{2}}{2t}}(L_{t}f(\xi)-L_{t}f(\zeta
))(h(\xi)-h(\zeta)):\xi\in\Sigma,f,h\in\mathscr{L}_{M}\right\}
,\label{defF*t}\\
\mathcal{H}_{*}^{t}  &  =\left\{  \omega^{*}_{t}(f,h,\xi)(\zeta)=t^{1/2}%
e^{-\frac{\|\xi-\zeta\|^{2}}{2t}}\Gamma_{t}(f,h)(\zeta):f,h\in\mathscr{L}_{M}%
,\xi\in\Sigma\right\}  . \label{defH*t}%
\end{align}
Then
\begin{align}
\mathcal{N}\left(  \mathcal{F}^{t},\varepsilon\right)   &  \le C_{d}V\left(
\frac{1}{\tau^{d}}+\left(  \frac{3C_{0}M^{2}}{\varepsilon}\right)
^{d}\right)  \left(  \frac{12M^{2}e^{-1/2}C^{*}(\Sigma)+\varepsilon
}{\varepsilon}\right)  ^{2N},\label{sizeFtuniform}\\
\mathcal{N}(\mathcal{H}^{t},\epsilon)  &  \le C_{d}V\left(  \frac{1}{\tau^{d}%
}+\left(  \frac{2C^{*}(\Sigma)e^{-1/2}M}{\varepsilon}\right)  ^{d}\right)
\left(  \frac{4Mt^{1/2}C^{*}(\Sigma)+\varepsilon}{\varepsilon}\right)
^{N},\label{sizeHtuniform}\\
\mathcal{N}(\mathcal{H}^{t}_{*},\epsilon)  &  \le C_{d}U\left(  \frac
{\epsilon}{3M^{2}\beta_{d}}\right)  \left(  \frac{12\alpha_{d}M+\epsilon
}{\epsilon}\right)  ^{2N}, \label{sizeH*}%
\end{align}
and for $t>0$ sufficiently small we have
\begin{align}
\mathcal{N}(\mathcal{F}^{t}_{*},\epsilon)  &  \le C_{d}U\left(  \frac
{\epsilon}{3C_{1}^{*}(\Sigma)M^{2}}\right)  \left(  \frac{24e^{-1/2}%
M^{2}t^{-1}C^{*}(\Sigma)+\varepsilon}{\varepsilon}\right)  ^{N}\left(
\frac{12e^{-1/2}M^{2}C_{1}^{*}(\Sigma)C^{*}(\Sigma)+\epsilon}{\epsilon
}\right)  ^{N}, \label{sizeF*}%
\end{align}
where $U(r)$ is as in \eqref{defU}. In the above inequalities, $C^{*}%
(\Sigma)=\displaystyle{\sup_{\xi\in\Sigma}\{\|\xi\|\}}$, $C^{*}(\Sigma)$ is a
positive constant depending on $\|D^{2}\mathrm{II}_{\Sigma}\|_{L^{\infty
}(\Sigma)}$, $C_{d},\alpha_{d},\beta_{d}$ are dimensional constants, $V$ is
the volume of $\Sigma$ and $C_{0}$ is the universal constant in Lemma
\ref{coveringphi}.
\end{lemma}

\begin{proof}
In order to prove \eqref{sizeFtuniform}, we estimate the difference
\begin{align*}
\left|  \phi_{t}(f,h,\xi)(\zeta)-\phi_{t}(\tilde{f},\tilde{h},\xi^{^{\prime}%
})(\zeta)\right|  ,
\end{align*}
for arbitrary $f,h,\tilde{f},\tilde{h}\in\mathscr{L}_{M}$ and $\xi
,\xi^{^{\prime}}\in\Sigma$. Note that
\begin{align*}
\left|  \phi_{t}(f,h,\xi,\zeta)-\phi_{t}(\tilde{f},\tilde{h},\xi^{^{\prime}%
},\zeta)\right|   &  \le\left|  \phi_{t}(f,h,\xi,\zeta)-\phi_{t}(\tilde
{f},h,\xi,\zeta)\right| \\
&  +\left|  \phi_{t}(\tilde{f},h,\xi,\zeta)-\phi_{t}(\tilde{f},\tilde{h}%
,\xi,\zeta)\right| \\
&  +\left|  \phi_{t}(\tilde{f},\tilde{h},\xi,\zeta)-\phi_{t}(\tilde{f}%
,\tilde{h},\xi^{^{\prime}},\zeta)\right|  .
\end{align*}
Note that
\begin{align*}
\left|  \phi_{t}(f,h,\xi,\zeta)-\phi_{t}(\tilde{f},h,\xi,\zeta)\right|   &
=t^{-1/2}e^{-\frac{\|\xi-\zeta\|^{2}}{2t}}\left|  h(\xi)-h(\zeta)\right|
\cdot\left|  f(\xi)-\tilde{f}(\xi)\right| \\
&  + t^{-1/2}e^{-\frac{\|\xi-\zeta\|^{2}}{2t}}\left|  h(\xi)-h(\zeta)\right|
\cdot|f(\zeta)-\tilde{f}(\zeta)|\\
&  \le2e^{-1/2}M\|f-\tilde{f}\|_{L^{\infty}(\Sigma)}.
\end{align*}
Similarly,
\begin{align*}
\left|  \phi_{t}(\tilde{f},h,\xi,\zeta)-\phi_{t}(\tilde{f},\tilde{h},\xi
,\zeta)\right|  \le2e^{-1/2}M\|h-\tilde{h}\|_{L^{\infty}(\Sigma)}.
\end{align*}
Finally, from Lemma \ref{coveringphi} we obtain
\begin{align*}
\left|  \phi_{t}(\tilde{f},\tilde{h},\xi,\zeta)-\phi_{t}(\tilde{f},\tilde
{h},\xi^{^{\prime}},\zeta)\right|  \le C(\tilde{f},\tilde{h})\|\xi
-\xi^{^{\prime}}\|\le C_{0}M^{2}\|\xi-\xi^{^{\prime}}\|.
\end{align*}
If we now turn to the computation of $L^{\infty}(\Sigma)$ covering numbers, we
obtain the bound
\begin{align}
\label{Nsquarebound}\mathcal{N}(\mathcal{F}^{t},\varepsilon)\le\mathcal{N}%
\left(  \mathscr{L}_{M},\frac{\epsilon}{6Me^{-1/2}}\right)  ^{2}\cdot
\sup_{f,h\in\mathscr{L}_{M}}\mathcal{N}\left(  \mathcal{F}^{t}_{f,h}%
,\frac{\epsilon}{3}\right)
\end{align}
and observe that for any $\delta>0$ we have the bounds
\begin{align}
\mathcal{N}\left(  \mathscr{L}_{M},\delta\right)   &  \le\left(  \frac
{2MC^{*}(\Sigma)+\delta}{\delta}\right)  ^{N},\label{boundLM}\\
\sup_{f,h\in\mathscr{L}_{M}}\mathcal{N}\left(  \mathcal{F}^{t}_{f,h}%
,\delta\right)   &  \le\mathcal{A}\left(  \Sigma,\frac{\delta}{C_{0}M^{2}%
}\right)  ,
\end{align}
where we have used that for any $u\in\mathscr{L}_{M}$ we can write
$u(\xi)=\langle\eta,\xi\rangle$ where $\eta\in B_{M}(0)$, i.e. the ball of
radius $M$ in $\mathbb{R}^{N}$ centered at the origin and $\xi\in\Sigma$, and
therefore $\mathcal{N}\left(  \mathscr{L}_{M},\delta\right)  \le
\mathcal{A}(B_{M}(0),\frac{\delta}{C^{*}(\Sigma)})$. On the other hand, it is
well known that for any $\upsilon>0$
\begin{align*}
\mathcal{A}\left(  B_{M}(0),\upsilon\right)  \le\left(  \frac{2M+\upsilon
}{\upsilon}\right)  ^{N},
\end{align*}
see for example \cite[Chapter 4]{Poll}. In view of Theorem \ref{coveringbound}
and \eqref{Nsquarebound}
\begin{align*}
\mathcal{N}\left(  \mathcal{F}^{t},\varepsilon\right)  \le\left(
\frac{12M^{2}e^{-1/2}C^{*}(\Sigma)+\varepsilon}{\varepsilon}\right)
^{2N}C_{d}V\left(  \frac{1}{\tau^{d}}+\left(  \frac{3C_{0}M^{2}}{\varepsilon
}\right)  ^{d}\right)  ,
\end{align*}
where $C_{d}$ is a dimensional constant, $V$ is the volume of $\Sigma$ and
$\tau$ is the reach of $\Sigma$. In order to prove \eqref{sizeHtuniform}, we
use the estimate
\begin{align}
\left|  \omega_{t}(\zeta,x)h(x)-\omega_{t}(\zeta^{^{\prime}},x)\tilde
{h}(x)\right|  \le\left|  \left(  \omega_{t}(\zeta,x)-\omega_{t}%
(\zeta^{^{\prime}},x)\right)  h(x)\right|  +\omega_{t}(\zeta^{^{\prime}%
},x)\left|  h(x)-\tilde{h}(x)\right| \nonumber\\
\le C^{*}(\Sigma)M e^{-1/2}\|x-x^{^{\prime}}\|+t^{1/2}\|h-\tilde
{h}\|_{L^{\infty}(\Sigma)}. \label{psihestimate}%
\end{align}
In order to obtain \eqref{psihestimate} we have used inequality
\eqref{bracketpsi}, $\|\psi_{t}\|_{L^{\infty}(\Sigma\times\Sigma)}=t^{1/2}$
and $\displaystyle{\sup_{h\in\mathscr{L}_{M}}\|h\|_{L^{\infty}(\Sigma)}\le
MC^{*}(\Sigma)}$. It follows that
\begin{align*}
\mathcal{N}\left(  \mathcal{H}^{t},\epsilon\right)   &  \le\mathcal{A}\left(
\Sigma,\frac{\varepsilon}{2C^{*}(\Sigma)e^{-1/2}M}\right)  \cdot
\mathcal{N}\left(  \mathscr{L}_{M},\frac{\varepsilon}{2t^{1/2}}\right) \\
&  \le\mathcal{A}\left(  \Sigma,\frac{\varepsilon}{2C^{*}(\Sigma)e^{-1/2}%
M}\right)  \cdot\mathcal{A}\left(  B_{M}(0),\frac{\varepsilon}{2t^{1/2}%
C^{*}(\Sigma)}\right) \\
&  \le C_{d}V\left(  \frac{1}{\tau^{d}}+\left(  \frac{2C^{*}(\Sigma)e^{-1/2}%
M}{\epsilon}\right)  ^{d}\right)  \left(  \frac{4Mt^{1/2}C^{*}(\Sigma
)+\epsilon}{\epsilon}\right)  ^{N},
\end{align*}
as claimed. For the proof of \eqref{sizeH*} we observe that
\begin{align*}
\left|  \omega^{*}_{t}(f,h,\xi)(\zeta)-\omega^{*}_{t}(\tilde{f},\tilde{h}%
,\xi^{^{\prime}})(\zeta)\right|   &  \le\left|  \omega^{*}_{t}(f,h,\xi
)(\zeta)-\omega^{*}_{t}(\tilde{f},h,\xi)(\zeta)\right| \\
&  +\left|  \omega^{*}_{t}(\tilde{f},h,\xi,\zeta)-\omega^{*}_{t}(\tilde
{f},\tilde{h},\xi,\zeta)\right| \\
&  +\left|  \omega^{*}_{t}(\tilde{f},\tilde{h},\xi)(\zeta)-\omega^{*}%
_{t}(\tilde{f},\tilde{h},\xi^{^{\prime}})(\zeta)\right|  .
\end{align*}
Observe that
\begin{align}
\left|  \omega^{*}_{t}(f,h,\xi)(\zeta)-\omega^{*}_{t}(\tilde{f},h,\xi
)(\zeta)\right|  \le\frac{2\|f-\tilde{f}\|_{L^{\infty}(\Sigma)}e^{-\frac
{\|\xi-\zeta\|^{2}}{2t}}}{t^{1/2}}\int_{\Sigma}e^{-\frac{\|\eta-\xi\|^{2}}%
{2t}}\left|  h(\eta)-h(\xi)\right|  d\mu(\eta).
\end{align}
It follows that
\begin{align}
\label{omega*1}\left|  \omega^{*}_{t}(f,h,\xi)(\zeta)-\omega^{*}_{t}(\tilde
{f},h,\xi)(\zeta)\right|  \le\alpha_{d}\|f-\tilde{f}\|_{L^{\infty}(\Sigma
)}\|h\|_{\mathrm{Lip}}\le2\alpha_{d}M\|f-\tilde{f}\|_{L^{\infty}(\Sigma)},
\end{align}
for some dimensional constant $\alpha_{d}$. Analogously we have
\begin{align}
\label{omega*2}\left|  \omega^{*}_{t}(\tilde{f},h,\xi,\zeta)-\omega^{*}%
_{t}(\tilde{f},\tilde{h},\xi,\zeta)\right|  \le2\alpha_{d}M\|h-\tilde
{h}\|_{L^{\infty}(\Sigma)}.
\end{align}
We also have
\begin{align}
\left|  \omega^{*}_{t}(\tilde{f},\tilde{h},\xi)(\zeta)-\omega^{*}_{t}%
(\tilde{f},\tilde{h},\xi^{^{\prime}})(\zeta)\right|   &  \le\frac
{e^{-1/2}\|\xi-\xi^{^{\prime}}\|}{t\theta_{t}(\xi)}\int_{\Sigma}%
e^{-\frac{\|\xi-\eta\|^{2}}{2t}}\left|  \tilde{f}(\eta)-\tilde{f}(\xi)\right|
\left|  \tilde{h}(\eta)-\tilde{h}(\xi)\right|  d\mu(\eta)\\
&  \le\beta_{d}e^{-1/2}M^{2}\|\xi-\xi^{^{\prime}}\|. \label{omega*3}%
\end{align}
Combining \eqref{omega*1}, \eqref{omega*2} and \eqref{omega*3} with an
argument similar to the one used to prove \eqref{sizeFtuniform} can be used to
prove \eqref{sizeH*}. For the proof \eqref{sizeF*}, we start with the
estimate
\begin{align*}
\left|  \phi^{*}_{t}(f,h,\xi)(\zeta)-\phi^{*}_{t}(\tilde{f},\tilde{h}%
,\xi^{^{\prime}})(\zeta)\right|   &  \le\left|  \phi^{*}_{t}(f,h,\xi
)(\zeta)-\phi^{*}_{t}(\tilde{f},h,\xi)(\zeta)\right| \\
&  +\left|  \phi^{*}_{t}(\tilde{f},h,\xi,\zeta)-\phi^{*}_{t}(\tilde{f}%
,\tilde{h},\xi,\zeta)\right| \\
&  +\left|  \phi^{*}_{t}(\tilde{f},\tilde{h},\xi)(\zeta)-\phi^{*}_{t}%
(\tilde{f},\tilde{h},\xi^{^{\prime}})(\zeta)\right|  .
\end{align*}
A simple estimate shows that
\begin{align}
\label{phi*1}\left|  \phi^{*}_{t}(f,h,\xi)(\zeta)-\phi^{*}_{t}(\tilde{f}%
,h,\xi)(\zeta)\right|  \le\frac{2Me^{-1/2}\|f-\tilde{f}\|_{L^{\infty}(\Sigma
)}}{t}.
\end{align}
Similarly
\begin{align}
\label{phi*2}\left|  \phi^{*}_{t}(\tilde{f},h,\xi,\zeta)-\phi^{*}_{t}%
(\tilde{f},\tilde{h},\xi,\zeta)\right|  \le2e^{-1/2}\|h-\tilde{h}%
\|_{L^{\infty}(\Sigma)}\|L_{t}\tilde{f}\|_{\mathrm{Lip}}.
\end{align}
Observe now that $\tilde{f}$ is a linear function that can be written as
$\tilde{f}(x)=\langle\eta,x\rangle$ for some $\eta$ with $\|\eta\|\le M$, and
therefore
\begin{align*}
\nabla^{2}_{\Sigma}\tilde{f}(x)=-\eta^{\perp}\mathrm{II}_{\Sigma}(x),
\end{align*}
where $\eta^{\perp}$ is the normal component of $\eta$ with respect to
$T_{x}\Sigma$. In particular, \eqref{Ltexpand} shows that
\begin{align*}
L_{t}\tilde{f}(x)=-\eta^{\perp}H_{\Sigma}(x)+R_{t}(x),
\end{align*}
where $H_{\Sigma}(x)$ is the mean curvature of $\Sigma$ at $x$ and $R(x)$ has
a bound of the form $\|R_{t}(x)\|_{L^{\infty}(\Sigma)}\le t^{1/2}%
M\|\nabla_{\Sigma}\mathrm{II}_{\Sigma}\|_{L^{\infty}(\Sigma)}$. It follows
that for $t>0$ small enough we have a bound of the form
\begin{align}
\label{Ltf}\|L_{t}\tilde{f}\|_{\mathrm{Lip}}\le MC^{*}_{1}(\Sigma).
\end{align}
We also have
\begin{align}
\left|  \phi^{*}_{t}(\tilde{f},\tilde{h},\xi)(\zeta)-\phi^{*}_{t}(\tilde
{f},\tilde{h},\xi^{^{\prime}})(\zeta)\right|   &  \le C_{0}\|L_{t}\tilde
{f}\|_{\mathrm{Lip}}\|\tilde{h}\|_{\mathrm{Lip}}\|\xi-\xi^{^{\prime}%
}\|\nonumber\\
&  \le C_{0}M^{2}C^{*}_{1}(\Sigma)\|\xi-\xi^{^{\prime}}\|. \label{phi*3}%
\end{align}
Combining \eqref{phi*1},\eqref{phi*2},\eqref{Ltf} and \eqref{phi*3},
inequality \eqref{sizeF*} follows easily (with an argument similar to the one
used to prove \eqref{sizeFtuniform}).

\end{proof}

\begin{lemma}
\label{newmaxlemma} The classes \eqref{defFt},\eqref{defHt},\eqref{defF*t} and
\eqref{defH*t} satisfy the following bounds
\begin{align}
\mathscr{M}_{\mathcal{F}^{t}}  &  =\sup_{u\in\mathcal{F}^{t}}\|u\|_{L^{\infty
}(\Sigma)}\le2t^{1/2}e^{-1}M^{2},\label{maxFt}\\
\mathscr{M}_{\mathcal{H}^{t}}  &  =\sup_{v\in\mathcal{H}^{t}}\|v\|_{L^{\infty
}(\Sigma)}\le t^{1/2}MC^{*}(\Sigma),\label{maxHt}\\
\mathscr{M}_{\mathcal{H}_{*}^{t}}  &  =\sup_{v\in\mathcal{H}_{*}^{t}%
}\|v\|_{L^{\infty}(\Sigma)}\le\beta_{d}M^{2} t^{1/2}. \label{maxH*t}%
\end{align}
In addition, for $t>0$ sufficiently small we have
\begin{align}
\mathscr{M}_{\mathcal{F}_{*}^{t}}\le t^{1/2}C^{*}_{1}(\Sigma)M^{2},
\label{maxF*t}%
\end{align}
where $C^{*}(\Sigma),C^{*}_{1}(\Sigma)$ and $\beta_{d}$ are as in Lemma
\ref{classesFtHt}.
\end{lemma}

\begin{convention}
\label{Cconvention} \emph{{ From now on, we will use the following convention
}}

\begin{itemize}
\item \emph{We will use $C^{*}(\Sigma)$ to denote the number
\begin{align*}
\sup_{\xi\in\Sigma}\{\|\xi\|\},
\end{align*}
up to multiplication by dimensional constants that may change from line line.
\newline}

\item \emph{We will use $C^{*}_{1}(\Sigma)$ to denote a constant that depends
on $L^{\infty}(\Sigma)$ bounds on $\mathrm{II}_{\Sigma}$ (second fundamental
form of $\Sigma$) and its derivatives. }
\end{itemize}

\emph{ }
\end{convention}

\subsection{Sample Version of the Carr\'{e} du Champ}

\label{secsamplecdc}

In this section we are still assuming that the distribution of the sample
$\xi_{1},\ldots,\xi_{n}$ in $\Sigma$ is \emph{uniform}, i.e., $d\mu
=d\mathrm{vol}_{\Sigma}$. In this case we know that $\lim_{t\rightarrow
0}t^{-d/2}\theta_{t}=(2\pi)^{d/2}$ uniformly in $L^{\infty}(\Sigma)$ and
therefore there exists $t_{0}>0$ such that for $0<t<t_{0}$ we have
\begin{align*}
2(2\pi)^{d/2}\ge\theta_{t}\ge\frac{1}{2}(2\pi)^{d/2}.
\end{align*}
If we let
\begin{align}
\lambda_{0}=\frac{(2\pi)^{d/2}}{4},
\end{align}
we have for $0< t<t_{0}$ the inequality
\begin{align}
\label{lowlambda0}\theta_{t}(x)\ge2t^{d/2}\lambda_{0},
\end{align}
which will be a convenient normalization for us in the sequel. The main lemma
in this section is the following.

{ }

\begin{lemma}
\label{lemmaGC}Let $\mathcal{J}$ be given by a class of functions of the form
\[
\varphi(f,x)(\cdot)=f(x,\cdot)
\]
for $x \in\Sigma$ and $f\in\mathcal{F}$ where $\mathcal{F}$ is a totally
bounded class of functions in $L^{\infty}(\Sigma\times\Sigma)$. Let
\[
C=\sup_{f\in\mathcal{F}}{\Vert f\Vert_{L^{\infty}(\Sigma\times\Sigma)}.}%
\]
Suppose also $0<t<t_{0}$ and $\varepsilon$ is small enough so that
\begin{equation}
\varepsilon t^{d+1/2}\lambda_{0}<C. \label{cond351}%
\end{equation}
Then $\mathcal{J}$ is totally bounded in $L^{\infty}(\Sigma)$ and
\begin{align}
&  {\Pr}^{\ast}\left\{  \sup_{f\in\mathcal{F}}\sup_{x\in\Sigma}\left\vert
t^{-1/2}\frac{\mu_{n}f(x,\cdot)}{\hat{\theta}_{t}(x)}-t^{-1/2}\frac{\mu
f(x,\cdot)}{\theta_{t}(x)}\right\vert \geq\varepsilon\right\} \label{GC1}\\
&  \leq2\mathcal{N}\left(  \mathcal{G}^{t},\frac{\varepsilon t^{d+1}%
\lambda_{0}^{2}}{4C}\right)  \exp\left(  -\frac{\varepsilon^{2}t^{2d+1}%
\lambda_{0}^{4}}{8C^{2}}n\right) \label{GC2}\\
&  2\mathcal{N}\left(  \mathcal{F},\frac{\varepsilon\lambda_{0}t^{(d+1)/2}}%
{4}\right)  \exp\left(  -\frac{\varepsilon^{2}\lambda_{0}^{2}t^{d+1}n}{8C^{2}%
}\right)  .
\end{align}

\end{lemma}

Before proving Lemma \ref{lemmaGC} we will prove the following elementary lemma.

\begin{lemma}
\label{manipulation1} Let $\xi,\zeta$ be positive random variables. For any
$\varepsilon>0$ we have
\begin{align*}
\Pr\left\{  \left|  \frac{1}{\xi}-\frac{1}{\zeta}\right|  \ge\varepsilon
\right\}  \le\Pr\left\{  |\zeta-\xi\vert\geq\frac{\varepsilon\xi^{2}%
}{1+\varepsilon\xi}\right\}  .
\end{align*}

\end{lemma}

\begin{proof}
Assume $0<\zeta<\xi$
\begin{align*}
\frac{1}{\zeta}-\frac{1}{\xi}  &  =\frac{\xi-\zeta}{\xi\zeta}=\frac{\xi-\zeta
}{\xi\zeta-\xi^{2}+\xi^{2}}\\
&  =\frac{\zeta-\xi}{\xi\left(  \zeta-\xi\right)  +\xi^{2}}\\
&  =\frac{|\zeta-\xi|}{-\xi|\zeta-\xi|+\xi^{2}},
\end{align*}
and from
\[
\frac{\left\vert \zeta-\xi\right\vert }{-\xi\left\vert \zeta-\xi\right\vert
+\xi^{2}}\leq\varepsilon
\]
we have
\[
\left\vert \zeta-\xi\right\vert \leq\frac{\varepsilon\xi^{2}}{1+\varepsilon
\xi}.
\]
For the case $0<\xi<\zeta$ we have
\[
\frac{1}{\xi}-\frac{1}{\zeta}=\frac{\zeta-\xi}{\xi(\zeta-\xi)+\xi^{2}}%
=\frac{|\xi-\zeta|}{\xi|\xi-\zeta|+\xi^{2}}.
\]
and $\frac{|\xi-\zeta|}{\xi|\xi-\zeta|+\xi^{2}}\geq\varepsilon$ implies
\[
(1+\varepsilon\xi)|\xi-\zeta|\geq(1-\varepsilon\xi)|\xi-\zeta|\geq
\varepsilon\xi^{2}.
\]

\end{proof}

\begin{proof}
[Proof of Lemma \ref{lemmaGC}]It is easy to prove from Lemma
\ref{manipulation1} and \eqref{lowlambda0} that for any $\delta>0$ we have%

\begin{align}
&  \Pr\left\{  \sup_{x\in\Sigma}\left\vert \frac{1}{\hat{\theta}_{t}(x)}%
-\frac{1}{\theta_{t}(x)}\right\vert \geq\delta\right\} \label{probtheta1}\\
&  \leq\Pr\left\{  \sup_{x\in\Sigma}\left\vert \theta_{t}(x)-\hat{\theta}%
_{t}(x)\right\vert \geq\frac{4\delta t^{d}\lambda_{0}^{2}}{1+2\delta
t^{d/2}\lambda_{0}}\right\}  . \label{probtheta2}%
\end{align}

Let us write
\begin{align*}
t^{-1/2}\left(  \frac{\mu_{n}f(x,\cdot)}{\hat{\theta}_{t}(x)}-\frac{\mu
f(x,\cdot)}{\theta_{t}(x)}\right)   &  =t^{-1/2}\mu_{n}f(x,\cdot)\left(
\frac{1}{\hat{\theta}_{t}(x)}-\frac{1}{\theta_{t}(x)}\right) \\
&  +t^{-1/2}\frac{1}{\theta_{t}(x)}\left[  \mu_{n}f(x,\cdot)-\mu
f(x,\cdot)\right]  .
\end{align*}
Thus
\begin{align*}
&  \Pr\left\{  t^{-1/2}\sup_{f\in\mathcal{F}}\sup_{x\in\Sigma}\left\vert
\frac{\mu_{n}f(x,\cdot)}{\hat{\theta}_{t}(x)}-\frac{\mu f(x,\cdot)}{\theta
_{t}(x)}\right\vert \geq\varepsilon\right\} \\
&  \leq\Pr\left\{  \sup_{x\in\Sigma}\left\vert \frac{1}{\hat{\theta}_{t}%
(x)}-\frac{1}{\theta_{t}(x)}\right\vert \geq\frac{\varepsilon}{2}\frac
{t^{1/2}}{C}\right\} \\
&  +\Pr\left\{  \sup_{f\in\mathcal{F}}\sup_{x\in\Sigma}\left\vert \mu
_{n}f(x,\cdot)-\mu f(x,\cdot)\right\vert \geq\varepsilon\lambda_{0}%
t^{(d+1)/2}\right\}  .
\end{align*}

Analyzing the first term using \eqref{cond351} and
\eqref{probtheta1}-\eqref{probtheta2} leads us to the inequality
\begin{align*}
&  \Pr\left\{  \sup_{x\in\Sigma}\left\vert \frac{1}{\hat{\theta}_{t}(x)}%
-\frac{1}{\theta_{t}(x)}\right\vert \geq\frac{\varepsilon}{2}\frac{t^{1/2}}%
{C}\right\} \\
&  \leq\Pr\left\{  \sup_{x\in\Sigma}\left\vert \theta_{t}(x)-\hat{\theta}%
_{t}(x)\right\vert \geq\frac{2\varepsilon t^{d+1/2}\lambda_{0}^{2}%
}{C+\varepsilon t^{d+1/2}\lambda_{0}}\right\} \\
&  \leq\Pr\left\{  \sup_{x\in\Sigma}\left\vert \theta_{t}(x)-\hat{\theta}%
_{t}(x)\right\vert \geq\frac{\varepsilon t^{d+1/2}\lambda_{0}^{2}}{C}\right\}
\\
\leq &  2\mathcal{N}\left(  \mathcal{G}^{t},\frac{\varepsilon t^{d+1}%
\lambda_{0}^{2}}{4C}\right)  \exp\left(  -\frac{\varepsilon^{2}t^{2d+1}%
\lambda_{0}^{4}}{8C^{2}}n\right)  ,
\end{align*}
where we have applied Lemma \ref{equicontinuity} to the class $\mathcal{G}%
^{t}$ (in particular we have used \eqref{boundGt}). Finally,
\begin{align*}
&  \Pr\left\{  \sup_{f\in\mathcal{F}}\sup_{x\in\Sigma}\left\vert \mu
_{n}f(x,\cdot)-\mu f(x,\cdot)\right\vert \geq\varepsilon\lambda_{0}%
t^{(d+1)/2}\right\} \\
&  \leq2\mathcal{N}\left(  \mathcal{F},\frac{\varepsilon\lambda_{0}%
t^{(d+1)/2}}{4}\right)  \exp\left(  -\frac{\varepsilon^{2}\lambda_{0}%
^{2}t^{(d+1)}n}{8C^{2}}\right)  .
\end{align*}

\end{proof}

In view of Lemma \ref{lemmaGC}, we introduce the following notation

\begin{definition}
Given a class of functions $\mathcal{F}$ as in the statement of Lemma
\ref{lemmaGC} and positive numbers $t,\varepsilon,M$ we define for compactness
of notation the following function
\begin{align}
Q_{t}(\mathcal{F},\varepsilon,C,n)  &  =2\mathcal{N}\left(  \mathcal{G}%
^{t},\frac{\varepsilon t^{d+1}\lambda_{0}^{2}}{4t^{1/2}C}\right)  \exp\left(
-\frac{\varepsilon^{2}t^{2d+1}\lambda_{0}^{4}}{8\left(  t^{1/2}C\right)  ^{2}%
}n\right) \label{Qdef}\\
&  +2\mathcal{N}\left(  \mathcal{F},\frac{\varepsilon\lambda_{0}t^{(d+1)/2}%
}{4}\right)  \exp\left(  -\frac{\varepsilon^{2}\lambda_{0}^{2}t^{d+1}%
n}{8\left(  t^{1/2}C\right)  ^{2}}\right)  .
\end{align}

\end{definition}

\begin{remark}
Because each one of theese function classes satisfies a bound of the form
(\ref{boundFt}), we include the $t^{1/2}$ factor in the expression for $C.$ \ 
\end{remark}

\begin{corollary}
If $\mathcal{F}$ is any of the classes of functions defined by \eqref{Ft},
\eqref{Gt}, \eqref{Ht}, \eqref{defFt},\eqref{defHt}, \eqref{defF*t},
\eqref{defF*t} we have%
\begin{equation}
{\Pr}^{\ast}\left\{  \sup_{x\in\Sigma}\left\vert t^{-1/2}\frac{\mu
_{n}f(x,\cdot)}{\hat{\theta}_{t}(x)}-t^{-1/2}\frac{\mu f(x,\cdot)}{\theta
_{t}(x)}\right\vert \geq\varepsilon\right\}  \leq Q_{t}(\mathcal{F}%
,\varepsilon,C,n).
\end{equation}
For \eqref{Ft}, the constant $C$ depends on $\Sigma$, $\left\Vert f\right\Vert
_{\mathrm{Lip}}$ and $\left\Vert h\right\Vert _{\mathrm{Lip}}$. \ For
\eqref{Gt}, $C$ depends on $\Sigma.$ \ For \eqref{Ht}, $C$ depends on
$\left\Vert h\right\Vert _{\infty}$ and $\Sigma.$ For \eqref{defFt},
\eqref{defHt}, \eqref{defF*t} and \eqref{defH*t}, $C$ depends on $\Sigma$, $M$
and dimensional constants.
\end{corollary}

As a corollary we obtain the rate of convergence in probability of the sample
Carr\'{e} du Champ to its expected value.

{ }

\begin{corollary}
\label{CDCn}

\begin{enumerate}
We have the following bounds

\item[(a)] Fixing $f$ and $h$ and letting
\[
K=K\left(  f,h\right)  =\min\left\{  \left(  \frac{2}{e}\right)  \Vert
f\Vert_{\mathrm{Lip}}\Vert h\Vert_{\mathrm{Lip}},\Vert f\Vert_{\text{A-t-}%
\mathrm{Lip}}\Vert h\Vert_{\mathrm{Lip}}\right\}  ,
\]
we have
\[
\Pr\left\{  \sup_{x\in\Sigma}\left\vert \hat{\Gamma}_{t}(f,h)(x)-\Gamma
_{t}(f,h)(x)\right\vert \geq\varepsilon\right\}  \leq Q_{t}(\mathcal{F}%
_{f,h}^{t},\varepsilon,K,n).
\]

\item[(b)] For the class $\mathscr{L}_{M}$ we have
\begin{align*}
\Pr\left\{  \sup_{f,h\in\mathscr{L}_{M}}\sup_{x\in\Sigma}\left\vert
\hat{\Gamma}_{t}(f,h)(x)-\Gamma_{t}(f,h)(x)\right\vert \geq\varepsilon
\right\}  \leq Q_{t}(\mathcal{F}^{t},\varepsilon,K,n),
\end{align*}
where $K=\frac{2M^{2}}{e}$.

\item[(c)] In addition, for $t>0$ sufficiently small we have
\begin{align*}
\Pr\left\{  \sup_{f,h\in\mathscr{L}_{M}}\sup_{x\in\Sigma}\left\vert
\hat{\Gamma}_{t}(L_{t}f,h)(x)-\Gamma_{t}(L_{t}f,h)(x)\right\vert
\geq\varepsilon\right\}  \leq Q_{t}(\mathcal{F}_{*}^{t},\varepsilon,K,n),
\end{align*}
where $K=C^{*}_{1}(\Sigma)M^{2}$, and $C^{*}_{1}(\Sigma)$ is as in Convention
\ref{Cconvention}.

\end{enumerate}
\end{corollary}

\begin{proof}
{Let us only prove (a) since the proofs of (b) and (c) are very similar}.
Recall that
\begin{align}
\hat{\Gamma}_{t}(L_{t},f,h)(x)  &  =\frac{1}{t}\frac{1}{\hat{\theta}_{t}}%
\sum_{j=1}^{n}e^{-\frac{\Vert x-\xi_{j}\Vert^{2}}{2t}}(f(\xi_{j}%
)-f(x))(h(\xi_{j})-h(x))\\
&  =\frac{t^{-1/2}}{\hat{\theta}_{t}(x)}\sum_{j=1}^{n}\phi_{t}(x,\xi
_{j})=t^{-1/2}\frac{\mu_{n}(\phi_{t}(x,\cdot))}{\hat{\theta}_{t}(x)},
\end{align}
where $\phi_{t}$ is given by the definition of the class $\mathcal{F}%
_{f,h}^{t}$ in \eqref{Ft} and therefore we can apply Lemma \ref{lemmaGC} to
the class $\mathcal{F}_{f,h}^{t}$ and use the bound
\[
\sup_{\phi\in\mathcal{F}_{f,h}^{t}}\Vert\phi\Vert_{L^{\infty}(\Sigma)}\leq
t^{1/2}K.
\]
which follows from \eqref{boundFt}.
\end{proof}

If we now consider the $t$-Laplacian $L_{t}$ and its sample version $\hat
{L}_{t}$, we see that the deviation of $\hat{L}_{t}$ from $L_{t}$ on a
function $h\in L^{\infty}(\Sigma)$ with $\Vert h\Vert_{L^{\infty}(\Sigma)}\leq
M$ simplifies to
\begin{align}
\hat{L}_{t}h(x)-L_{t}h(x)  &  =\frac{2}{t\hat{\theta}_{t}(x)n}\sum_{j=1}%
^{n}e^{-\frac{\Vert\xi_{j}-x\Vert^{2}}{2t}}(h(\xi_{j}%
)-h(x))\label{cancellation1}\\
&  -\frac{2}{t\theta_{t}(x)}\int_{\Sigma}e^{-\frac{\Vert\xi-x\Vert^{2}}{2t}%
}(h(\xi)-h(x))d\mu(\xi)\\
&  =\frac{2}{\hat{\theta}_{t}(x)tn}\sum_{j=1}^{n}e^{-\frac{\Vert\xi_{j}%
-x\Vert^{2}}{2t}}h(\xi_{j})-\frac{2}{t\theta_{t}(x)}\int_{\Sigma}%
e^{-\frac{\Vert\xi-x\Vert^{2}}{2t}}h(\xi)d\mu(\xi)\\
&  =\frac{2\mu_{n}\psi_{t}(x,\cdot)h(\cdot)}{t^{3/2}\hat{\theta}_{t}(x)}%
-\frac{2\mu\psi_{t}(x,\cdot)h(\cdot)}{t^{3/2}\theta_{t}(x)}.
\label{cancellationlast}%
\end{align}
Observe that
\begin{equation}
\Pr\left\{  \sup_{x\in\Sigma}\left\vert \hat{L}_{t}h(x)-L_{t}h(x)\right\vert
\geq\varepsilon\right\}  \leq\Pr\left\{  t^{-1/2}\sup_{\eta\in\mathcal{H}%
_{h}^{t}}\left\vert \frac{\mu_{n}\eta}{\hat{\theta}_{t}}-\frac{\mu\eta}%
{\theta_{t}}\right\vert \geq t\varepsilon\right\}  . \label{unifdevL}%
\end{equation}
We have obtained

{ }

\begin{corollary}
We have the following bounds \label{BN}

\begin{enumerate}
\item[(a)] Fix a function $h\in L^{\infty}(\Sigma)$. If we set $C=\left\Vert
h\right\Vert _{L^{\infty}}$ we have
\[
\Pr\left\{  \sup_{x\in\Sigma}\left\vert \hat{L}_{t}h(x)-L_{t}h(x)\right\vert
\geq\varepsilon\right\}  \leq Q_{t}(\mathcal{H}_{h}^{t},\varepsilon t,C,n).
\]

\item[(b)] For the class $\mathscr{L}_{M}$ we have the estimate
\begin{align*}
\Pr\left\{  \sup_{h\in\mathscr{L}_{M}}\sup_{x\in\Sigma}\left\vert \hat{L}%
_{t}h(x)-L_{t}h(x)\right\vert \geq\varepsilon\right\}  \leq Q_{t}%
(\mathcal{H}^{t},\varepsilon t,C,n),
\end{align*}
where $C=MC^{*}(\Sigma)$.

\item[(c)] In addition, we have for every $t>0$ the estimate
\begin{align*}
\Pr\left\{  \sup_{f,h\in\mathscr{L}_{M}}\sup_{x\in\Sigma}\left\vert \hat
{L}_{t}\left(  \Gamma_{t}(f,h)\right)  (x)-L_{t}\left(  \Gamma_{t}%
(f,h)\right)  (x)\right\vert \geq\varepsilon\right\}  \leq Q_{t}%
(\mathcal{H}^{t}_{*},\varepsilon t,C,n),
\end{align*}
where $C=\beta_{d}M^{2}$, where $\beta_{d}>0$ is a dimensional constant.
\end{enumerate}
\end{corollary}

\begin{proof}
The proof of (a) follows from combining Lemma \ref{lemmaGC} with
\eqref{unifdevL} and the fact that
\begin{align*}
\sup_{\eta\in\mathcal{H}_{h}^{t}}\Vert\eta\Vert_{L^{\infty}(\Sigma)}\leq
t^{1/2}\Vert h\Vert_{L^{\infty}(\Sigma)}=t^{1/2}C.
\end{align*}
The proofs of (b) and (c) are similar to the proof of (a) but make use of
Lemmas \eqref{classesFtHt} and \eqref{newmaxlemma}.

\end{proof}

\subsection{Subexponential Decay and Almost Sure Convergence.}

\label{secsubexp}

The goal of this subsection is to demonstrate that the decay rate for the
quantities $Q_{t}(\mathcal{F},\varepsilon,M,n)$ implies the almost sure
convergence. \ \ We illustrate the Borel-Cantelli type proof in this section
for the purpose of introducing some notation. This notation will be used in
later sections.

{ }

\begin{theorem}
\label{extrinsic variance for Gamma} Consider the metric measure space
$(\Sigma,\Vert\cdot\Vert,d\mathrm{vol}_{\Sigma})$ where $\Sigma^{d}%
\subset\mathbb{R}^{N}$ is a smooth closed embedded submanifold. Suppose that
we have a uniformly distributed i.i.d. sample $\{\xi_{1},\ldots,\xi_{n}\}$ of
points from $\Sigma$. For $\sigma>0,$ let
\begin{equation}
t_{n}=n^{-\frac{1}{2d+\sigma}}. \label{choicetn}%
\end{equation}
Then

\begin{enumerate}
\item[(a)] For fixed $f,h\in\mathrm{Lip}(\Sigma)$ we have
\begin{align*}
\sup_{\xi\in\Sigma}\left\vert \hat{\Gamma}(L_{t_{n}},f,h)(\xi)-\Gamma
(L_{t_{n}},f,h)(\xi)\right\vert \overset{\text{a.s.}}{\longrightarrow}0
\end{align*}
as $n\rightarrow\infty.$

\item[(b)] For the class $\mathscr{L}_{M}$ we have
\begin{align*}
\sup_{f,h\in\mathscr{L}_{M}}\sup_{\xi\in\Sigma}\left\vert \hat{\Gamma
}(L_{t_{n}},f,h)(\xi)-\Gamma(L_{t_{n}},f,h)(\xi)\right\vert
\overset{\text{a.s.}}{\longrightarrow}0.
\end{align*}

\item[(c)] In addition we have
\begin{align*}
\sup_{f,h\in\mathscr{L}_{M}}\sup_{\xi\in\Sigma}\left\vert \hat{\Gamma
}(L_{t_{n}},L_{t_{n}}f,h)(\xi)-\Gamma(L_{t_{n}},L_{t_{n}}f,h)(\xi)\right\vert
\overset{\text{a.s.}}{\longrightarrow}0.
\end{align*}

\end{enumerate}
\end{theorem}

\begin{proof}
For part (a), if we fix $n$ and $\varepsilon,$ we have
\begin{align*}
\Pr\left\{  \sup_{x\in\Sigma}\left\vert \hat{\Gamma}(L_{t_{n}},f,h)(x)-\Gamma
(L_{t_{n}},f,h)(x)\right\vert \geq\varepsilon\right\}  \leq Q_{t_{n}%
}(\mathcal{F}_{f,h}^{t_{n}},\varepsilon,K,n).
\end{align*}
where $K=K(f,h)$ as in Corollary \ref{CDCn}. Now plugging in the expression
for $t_{n}$ we observe a bound of the form
\begin{equation}
Q_{t_{n}}(\mathcal{F}_{f,h}^{t_{n}},\varepsilon,K)\leq p\left(  n^{\frac
{1}{2\left(  2d+\sigma\right)  }},\frac{1}{\varepsilon}\right)  \exp\left(
-c_{1}\varepsilon^{2}n^{\sigma/(2d+\sigma)}\right)  \label{subexpbound}%
\end{equation}
where $p$ is a fixed polynomial bound and $c_{1}>0$ is a constant. Thus with
$\varepsilon$ fixed, we have
\[
\sum_{n=1}^{\infty}Q_{t_{n}}(\mathcal{F}_{f,h}^{t_{n}},\varepsilon
,K,n)<\infty.
\]
Applying the Borel-Cantelli Lemma gives the almost sure convergence. The proof
of part (b) is analogous, since we have the estimate
\begin{align*}
\Pr\left\{  \sup_{f,h\in\mathscr{L}_{M}}\sup_{x\in\Sigma}\left\vert
\hat{\Gamma}(L_{t_{n}},f,h)(x)-\Gamma(L_{t_{n}},f,h)(x)\right\vert
\geq\varepsilon\right\}  \leq Q_{t_{n}}(\mathcal{F}^{t_{n}},\varepsilon,K,n),
\end{align*}
with $K=\frac{2}{e}M^{2}$. Observe that $Q_{t_{n}}(\mathcal{F}^{t_{n}%
},\varepsilon,K,n)$ satisfies a bound of the form \eqref{subexpbound}, but
this time the coefficients of the polynomial $p$ depend on $\Sigma$, $M$ and
$N$ as seen in Lemma \ref{classesFtHt}. Part (c) follows from the estimate
\begin{align*}
\Pr\left\{  \sup_{f,h\in\mathscr{L}_{M}}\sup_{x\in\Sigma}\left\vert
\hat{\Gamma}(L_{t},L_{t}f,h)(x)-\Gamma(L_{t},L_{t}f,h)(x)\right\vert
\geq\varepsilon\right\}  \leq Q_{t}(\mathcal{F}^{t}_{*},\varepsilon,K,n),
\end{align*}
where $K=M^{2}C^{*}_{1}(\Sigma)$.
\end{proof}

Now we see that for almost sure convergence, one requires a bound of the form
(\ref{subexpbound}). For this reason, we introduce notation for use in the
sequel: consider a function
\begin{equation}
\mathcal{Q}:\mathbb{R}^{+}\times\mathbb{R}^{+}\times\mathbb{N}\rightarrow
\mathbb{R}^{+}. \label{qbcbeta}%
\end{equation}
We say that $\mathcal{Q}(t,\varepsilon,n)\in O_{BC}(\beta)$ if
\begin{equation}
\sum_{n=1}^{\infty}\mathcal{Q}(n^{-\frac{1}{\beta+\sigma}},\varepsilon
,n)<\infty\label{obcbeta}%
\end{equation}
for all $\sigma,\varepsilon>0.$ Clearly the definition gives
\[
O_{BC}(\beta)\subset O_{BC}(\beta^{\prime})
\]
for $\beta^{\prime}>\beta.$ We also observe that
\[
O_{BC}(\beta)+O_{BC}(\beta^{\prime})\in O_{BC}\left(  \max\left\{  \beta
,\beta^{\prime}\right\}  \right)  .
\]

\begin{lemma}
\label{probdecayFGH} For the classes of functions defined by \eqref{Ft},
\eqref{Gt}, \eqref{Ht}, \eqref{defFt}, \eqref{defHt},\eqref{defF*t},
\eqref{defH*t} and fixed $\varepsilon,M>0$ we have
\[
Q_{t}(\mathcal{F},\varepsilon,M,n)\in O_{BC}(2d).
\]

\end{lemma}

\begin{proof}
Plugging in
\[
t=n^{-\frac{1}{\beta+\sigma}}%
\]
the dominant exponential factor in (\ref{Qdef}) becomes
\[
\exp\left(  -\frac{\varepsilon^{2}\lambda_{0}^{4}}{8C}n^{\frac{\beta
+\sigma-2d}{\beta+\sigma}}\right)  .
\]
Clearly, for
\[
\beta+\sigma-2d>0
\]
we have%
\[
\sum_{n=1}^{\infty}\mathcal{Q}(n^{-\frac{1}{\beta+\sigma}},\varepsilon
,n)<\infty.
\]

\end{proof}

We record the following, which is clear from the definitions and Corollary
\ref{coveringFG}.

\begin{corollary}
\label{BC increase}Given any of the classes above, let
\[
\mathcal{Q}(t,\varepsilon,n)=Q_{t}(\mathcal{F},\varepsilon t^{\alpha
},Kt^{\delta},n).
\]
Then
\[
\mathcal{Q}(t,\varepsilon,n)\in O_{BC}(2d+2\alpha-2\delta).
\]

\end{corollary}

{ }

\subsection{Proof of Theorem \ref{extrinsic variance}}

\label{secproofLD} We will only carry out the proof of part (b). The proof of
part (a) is analogous. Recall that
\[
\Gamma_{2}(L_{t},f,f)=\frac{1}{2}\left(  L_{t}\left(  \Gamma_{t}(f,f)\right)
-2\Gamma_{t}(L_{t}f,f)\right)  ,
\]
and that from Remark \ref{schemCDC2}
\[
\hat{\Gamma}_{2}(L_{t},f,f)=\frac{1}{2}\left(  \hat{L}_{t}\left(  \hat{\Gamma
}_{t}(f,f)\right)  -2\hat{\Gamma}_{t}(\hat{L}_{t}f,f)\right)  ,
\]
so we will start by estimating the difference $\hat{L}_{t}\left(  \hat{\Gamma
}_{t}(f,f)\right)  -L_{t}\left(  \Gamma_{t}(f,f)\right)  $ which we write as
\begin{align}
\left[  \hat{L}_{t}(\hat{\Gamma}_{t}(f,f))-L(\Gamma_{t}(f,f))\right]  (x)=  &
\hat{L}_{t}\left(  \hat{\Gamma}_{t}(f,f)-\Gamma_{t}(f,f)\right)
(x)\label{split1}\\
&  +\left(  \hat{L}_{t}-L_{t}\right)  \Gamma_{t}(f,f)(x)\\
&  =A_{1}(x)+A_{2}(x),
\end{align}
and observe that
\begin{equation}
\left\Vert A_{1}\right\Vert _{\infty}=\sup_{x\in\Sigma}\left\vert \hat{L}%
_{t}\left(  \hat{\Gamma}_{t}(f,f)(x)-\Gamma_{t}(f,f)(x)\right)  \right\vert
\leq\frac{4}{t}\sup_{f\in\mathscr{L}_{M}}\sup_{\xi\in\Sigma}\left\vert
\hat{\Gamma}_{t}(f,f)(\xi)-\Gamma_{t}(f,f)(\xi)\right\vert . \label{A1}%
\end{equation}

In order to estimate $A_{2}=\left(  \hat{L}_{t}-L_{t}\right)  (\Gamma
_{t}(f,f))$ we make use of {Corollary \ref{BN}.} We now estimate the
difference
\begin{align}
\hat{\Gamma}_{t}(\hat{L}_{t}f,f)(x)-\Gamma_{t}(L_{t}f,f)(x)  &  =\hat{\Gamma
}_{t}(\hat{L}_{t}f-L_{t}f,f)(x)+\left(  \hat{\Gamma}_{t}-\Gamma_{t}\right)
(L_{t}f,f)(x)\label{split2}\\
&  =A_{3}(x)+A_{4}(x),
\end{align}
and we note
\begin{align}
\left\Vert A_{3}\right\Vert _{\infty}=  &  \sup_{x\in\Sigma}\left\vert
\hat{\Gamma}_{t}(\hat{L}_{t}f-L_{t}f,f)(x)\right\vert \\
&  \leq\sup_{x\in\Sigma}\frac{2}{t}\left(  \sup_{\xi\in\Sigma}\left\vert
\hat{L}_{t}(f)(\xi)-L_{t}(f)(\xi)\right\vert \right)  \frac{1}{n\hat{\theta
}_{t}(x)}\sum_{j=1}^{n}e^{-\frac{\Vert\xi_{j}-x\Vert^{2}}{2t}}|f(\xi
_{j})-f(x)|\\
&  \leq2\frac{e^{-1/2}\Vert f\Vert_{\mathrm{Lip}}}{t^{1/2}}\sup_{x\in\Sigma
}\frac{1}{\hat{\theta}_{t}(x)}\left(  \sup_{\xi\in\Sigma}\left\vert (\hat
{L}_{t}(f)(\xi)-L_{t}(f)(\xi))\right\vert \right) \nonumber\\
&  \leq2\frac{e^{-1/2}M}{t^{1/2}}\sup_{x\in\Sigma}\frac{1}{\hat{\theta}%
_{t}(x)}\left(  \sup_{f\in\mathscr{L}_{M}}\sup_{\xi\in\Sigma}\left\vert
(\hat{L}_{t}(f)(\xi)-L_{t}(f)(\xi))\right\vert \right)  \label{term3}%
\end{align}
where we have used the fact that for functions $u,v$
\begin{align}
\left\vert \hat{\Gamma}_{t}(u,v)(x)\right\vert  &  \leq\frac{2}{t}\sup_{\xi
\in\Sigma}|u(\xi)|\left(  \frac{1}{n\hat{\theta}_{t}(x)}\sum_{j=1}%
^{n}e^{-\frac{\Vert x-\xi_{j}\Vert^{2}}{2t}}|v(x)-v(\xi_{j})|\right) \\
&  \leq\frac{2}{t}\sup_{\xi\in\Sigma}|u(\xi)|\left(  \frac{\Vert
v\Vert_{\mathrm{Lip}}}{n\hat{\theta}_{t}(x)}\sup_{\rho>0}\rho e^{-\frac
{\rho^{2}}{2t}}\right)  ,
\end{align}
and
\[
\sup_{\rho>0}\rho e^{-\frac{\rho^{2}}{2t}}=t^{1/2}e^{-1/2}\leq t^{1/2}.
\]
In view of Lemma \eqref{lemmaGC}, we will write the bound \eqref{term3} on
$\left\Vert A_{3}\right\Vert _{\infty}$ as
\begin{align}
\left\Vert A_{3}\right\Vert _{\infty}  &  \leq2\frac{e^{-1/2}M}{t^{1/2}%
\theta_{t}}\sup_{x\in\Sigma}\frac{1}{\theta_{t}(x)}\left(  \sup_{f\in
\mathscr{L}_{M}}\sup_{\xi\in\Sigma}\left\vert (\hat{L}_{t}(f)(\xi
)-L_{t}(f)(\xi))\right\vert \right) \\
&  +\frac{2e^{-1/2}M}{t^{1/2}}\sup_{x\in\Sigma}\left(  \frac{1}{\hat{\theta
}_{t}(x)}-\frac{1}{\theta_{t}(x)}\right)  \left(  \sup_{f\in\mathscr{L}_{M}%
}\sup_{\xi\in\Sigma}\left\vert (\hat{L}_{t}(f)(\xi)-L_{t}(f)(\xi))\right\vert
\right)  . \label{term3.1}%
\end{align}
In order to estimate $|A_{4}|$, i.e.
\[
\left\Vert A_{4}\right\Vert _{\infty}= \sup_{x\in\Sigma}\left\vert \left(
\hat{\Gamma}_{t}-\Gamma_{t}\right)  (L_{t}(f,f))(x)\right\vert ,
\]
we will make use of {Corollary \ref{CDCn}}. \newline

Observe now that the terms $A_{1},A_{2},A_{3},A_{4}$ above are random
variables with expressions of the form $A_{i}=A_{i}(f,x)$, let $A_{i}^{\ast
}=\sup_{f\in\mathscr{L}_{M}}\sup_{x\in\Sigma}|A_{i}(f,x)|$ for $i=1,2,3,4$. We
now have from \eqref{split1} and \eqref{split2}
\begin{align}
&  \Pr\left\{  \sup_{f\in\mathscr{L}_{M}}\sup_{x\in\Sigma}\left\vert
\hat{\Gamma}_{2}(L_{t},f,f)-\Gamma_{2}(L_{t},f,f)\right\vert (x)\geq
\varepsilon\right\} \\
&  \leq\Pr\left\{  A^{*}_{1}\geq\frac{\varepsilon}{4}\right\}  +\Pr\left\{
A^{*}_{2}\geq\frac{\varepsilon}{4}\right\}  +\Pr\left\{  A^{*}_{3}\geq
\frac{\varepsilon}{4}\right\}  +\Pr\left\{  A^{*}_{4}\geq\frac{\varepsilon}%
{4}\right\} \\
&  =P_{1}+P_{2}+P_{3}+P_{4}.
\end{align}
From \eqref{A1} and Corollary \ref{CDCn} we have
\begin{align}
P_{1}  &  \leq\Pr\left\{  \sup_{f\in\mathscr{L}_{M}} \sup_{\xi\in\Sigma
}\left\vert \hat{\Gamma} _{t}(f,f)(\xi)-\Gamma_{t}(f,f)(\xi)\right\vert
\geq\frac{t\varepsilon} {16}\right\}  \leq Q_{t}\left(  \mathcal{F}_{f}%
^{t},\frac{t\varepsilon} {16},2e^{-1}M^{2},n\right)  ,\\
&  \in O_{BC}\left(  2d+2\right)  . \label{boundP1}%
\end{align}

The statement about the convergence order follows from\ Corollary
\ref{BC increase}.

For $A_{2}=\left(  \hat{L}_{t}-L_{t}\right)  (\Gamma_{t}(f,f))$ we apply part
(c) of Corollary \ref{BN} together with \eqref{maxH*t} and obtain
\begin{equation}
P_{2}\leq Q_{t}\left(  \mathcal{H}^{t}_{*},\varepsilon t,\beta_{d}%
M^{2},n\right)  \in O_{BC}\left(  2d+2\right)  . \label{boundP2}%
\end{equation}

Using \eqref{term3.1} we have
\begin{align}
P_{3}  &  \leq\Pr\left\{  \frac{\varepsilon}{8}\leq2\frac{e^{-1/2}M}%
{t^{1/2}\theta_{t}}\sup_{x\in\Sigma}\frac{1}{\theta_{t}(x)}\left(  \sup
_{f\in\mathscr{L}_{M}}\sup_{\xi\in\Sigma}\left\vert (\hat{L}_{t}(f)(\xi
)-L_{t}(f)(\xi))\right\vert \right)  \right\} \\
&  +\Pr\left\{  \frac{\varepsilon}{8}\leq\frac{2e^{-1/2}M}{t^{1/2}}\sup
_{x\in\Sigma}\left|  \frac{1}{\hat{\theta}_{t}(x)}-\frac{1}{\theta_{t}%
(x)}\right|  \left(  \sup_{f\in\mathscr{L}_{M}}\sup_{\xi\in\Sigma}\left\vert
(\hat{L}_{t}(f)(\xi)-L_{t}(f)(\xi))\right\vert \right)  \right\}
\end{align}
and choosing $t>0$ small enough we obtain from \eqref{lowlambda0} and
Corollary \ref{BN} the estimate
\begin{align}
&  \Pr\left\{  \frac{\varepsilon}{8}\leq2\frac{e^{-1/2}M}{t^{1/2}\theta_{t}%
}\left(  \sup_{f\in\mathscr{L}_{M}}\sup_{\xi\in\Sigma}\left\vert (\hat{L}%
_{t}(f)(\xi)-L_{t}(f)(\xi))\right\vert \right)  \right\} \\
&  \leq\Pr\left\{  \frac{\varepsilon}{8}\frac{\lambda_{0}t^{\frac{d+2}{2}%
}e^{1/2}}{M}\leq\sup_{f\in\mathscr{L}_{M}}\sup_{\xi\in\Sigma}\left\vert
(\hat{L}_{t}(f)(\xi)-L_{t}(f)(\xi))\right\vert \right\} \\
&  \leq Q_{t}\left(  \mathcal{H}^{t},\frac{\varepsilon}{8}\frac{\lambda
_{0}t^{\frac{d+1}{2}}e^{1/2}}{M}t,C,n\right)  \in O_{BC}(3d+3),
\end{align}
where $C=C^{*}(\Sigma)M\ge\sup_{f\in\mathscr{L}_{M}}\|f\|_{L^{\infty}(\Sigma
)}$. Here the increased order is $t^{\frac{d+1}{2}}$ and $t$ which is
$2d+d+1+2.$ \ On the other hand, we will make use of the estimate
\[
\sup_{f\in\mathscr{L}_{M}}\sup_{\xi\in\Sigma}\left\vert (\hat{L}_{t}%
(f)(\xi)-L_{t}(f)(\xi))\right\vert \leq\frac{4}{t}\sup_{f\in\mathscr{L}_{M}%
}\Vert f\Vert_{L^{\infty}}\le\frac{4C^{*}(\Sigma)M}{t}=\frac{4C}{t},
\]
for any $t>0$ to obtain
\begin{align}
&  \Pr\left\{  \frac{\varepsilon}{8}\leq\frac{2e^{-1/2}M}{t^{1/2}}\sup
_{x\in\Sigma}\left|  \frac{1}{\hat{\theta}_{t}(x)}-\frac{1}{\theta_{t}%
(x)}\right|  \left(  \sup_{f\in\mathscr{L}_{M}}\sup_{\xi\in\Sigma}\left\vert
(\hat{L}_{t}(f)(\xi)-L_{t}(f)(\xi))\right\vert \right)  \right\} \\
&  \leq\Pr\left\{  \frac{\varepsilon t^{3/2}e^{1/2}}{64M}\leq\sup_{x\in\Sigma
}\left\vert \frac{1}{\hat{\theta}_{t}(x)}-\frac{1}{\theta_{t}(x)}\right\vert
\right\}
\end{align}
and from \eqref{probtheta1}-\eqref{probtheta2} and Lemma \ref{hoeffding} we
observe that
\[
\Pr\left\{  \frac{\varepsilon}{8}\leq\frac{2e^{-1/2}M}{t^{1/2}}\sup
_{x\in\Sigma}\left|  \frac{1}{\hat{\theta}_{t}}-\frac{1}{\theta_{t}}\right|
\left(  \sup_{f\in\mathscr{L}_{M}}\sup_{\xi\in\Sigma}\left\vert (\hat{L}%
_{t}(f)(\xi)-L_{t}(f)(\xi))\right\vert \right)  \right\}  \in O_{BC}(d+3).
\]
We then conclude that
\begin{equation}
P_{3}\in O_{BC}(3d+3). \label{boundP3}%
\end{equation}
Finally, in order to estimate $P_{4}$, we use part (c) of Lemma \ref{CDCn},
namely
\begin{align*}
\Pr\left\{  \sup_{f\in\mathscr{L}_{M}}\sup_{x\in\Sigma}\left\vert \hat{\Gamma
}_{t}(L_{t}f,f)(x)-\Gamma_{t}(L_{t}f,f)(x)\right\vert \geq\varepsilon\right\}
\leq Q_{t}(\mathcal{F}^{t}_{*},\varepsilon,K,n),
\end{align*}
where $K=C_{1}^{*}(\Sigma)M^{2}$ (from \eqref{maxF*t}) and hence
\begin{equation}
P_{4}=\Pr\left\{  A^{*}_{4} \geq\frac{\varepsilon}{4}\right\}  \leq
Q_{t}(\mathcal{F}^{t}_{*},\frac{\varepsilon}{4},K,n)\in O_{BC}(2d).
\label{boundP4}%
\end{equation}
Using again
\[
\Pr\left\{  \left\vert \hat{\Gamma}_{2}(L_{t},f,f)-{\Gamma}_{2}(L_{t}%
,f,f)\right\vert \geq\varepsilon\right\}  \leq P_{1}+P_{2}+P_{3}+P_{4},
\]
and from \eqref{boundP1},\eqref{boundP2},\eqref{boundP3} and \eqref{boundP4}
we have
\begin{equation}
\Pr\left\{  \left\vert \hat{\Gamma}_{2}(L_{t},f,f)-{\Gamma}_{2}(L_{t}%
,f,f)\right\vert \geq\varepsilon\right\}  \in O_{BC}(3d+3),
\label{gamma2converge}%
\end{equation}
and it follows from the Borel-Cantelli argument given in Section
\ref{secsubexp} that for any sequence of the form $t_{n}=n^{-\gamma}$ where
$\gamma=\frac{1}{3d+3+\sigma}$ and $\sigma$ is any positive number we have
\[
\sup_{\xi\in\Sigma}\left\vert \hat{\Gamma}_{2}(L_{t_{n}},f,f)-{\Gamma}%
_{2}(L_{t_{n}},f,f)\right\vert \overset{\text{a.s.}}{\longrightarrow}0.
\]
This proves Theorem B.

Now with the convergence for each fixed function $f$ we can prove Corollary
\ref{corunif} (see section \ref{summary}).

\begin{proof}
[Proof of Corollary \ref{corunif}]We work on a compact smooth submanifold of
Euclidean space. \ With the ambient distance function, there is no cut locus,
and the set of functions given by (recall (\ref{defBIGF}))
\[
\mathcal{R=}\left\{  F_{x,y}:\left(  x,y\right)  \in\Sigma\times
\Sigma\right\}  ,
\]
is uniformly bounded in $C^{5}.$ \ \ The map
\begin{align*}
\Sigma\times\Sigma &  \rightarrow C^{5}(\Sigma)\\
(x,y)  &  \rightarrow F_{x,y}%
\end{align*}
is a Lipschitz map. \ It follows that we can take a finite $\delta$-net
$\mathcal{G}$ with respect to the $L^{\infty}$ topology, and the net size will
grow at worst polynomially. \ \ \ That is for a given
\[
f\in\mathcal{R}\text{ }%
\]
there exists
\[
f^{\ast}\in\mathcal{G}%
\]
such that
\[
\left\Vert f-f^{\ast}\right\Vert _{L^{\infty}}<\delta.\text{ \ }%
\]
The constant $\delta$ will be chosen below. \ We want to estimate the
probability
\[
P=\Pr\left\{  \sup_{f\in\mathcal{R}}\sup_{\xi\in\Sigma}\left\vert \hat{\Gamma
}_{2}(L_{t},f,f)\left(  \xi\right)  -{\Gamma}_{2}(\Delta_{g},f,f)\left(
\xi\right)  \right\vert \geq\varepsilon\right\}  .
\]
We have%
\begin{align*}
\hat{\Gamma}_{2}(L_{t},f,f)\left(  \xi\right)  -{\Gamma}_{2}(\Delta
_{g},f,f)\left(  \xi\right)   &  =\hat{\Gamma}_{2}(L_{t},f,f)\left(
\xi\right)  -\hat{\Gamma}_{2}(L_{t},f^{\ast},f^{\ast})\left(  \xi\right) \\
&  +\hat{\Gamma}_{2}(L_{t},f^{\ast},f^{\ast})\left(  \xi\right)  -{\Gamma}%
_{2}(L_{t},f^{\ast},f^{\ast})\left(  \xi\right) \\
&  +{\Gamma}_{2}(L_{t},f^{\ast},f^{\ast})\left(  \xi\right)  -{\Gamma}%
_{2}(L_{t},f,f)\left(  \xi\right) \\
&  +{\Gamma}_{2}(L_{t},f,f)\left(  \xi\right)  -{\Gamma}_{2}(\Delta
_{g},f,f)\left(  \xi\right)  .
\end{align*}
Thus,
\begin{align*}
P  &  \leq\Pr\left\{  \sup_{f\in\mathcal{R}}\sup_{\xi\in\Sigma}\left\vert
\hat{\Gamma}_{2}(L_{t},f,f)\left(  \xi\right)  -\hat{\Gamma}_{2}(L_{t}%
,f^{\ast},f^{\ast})\right\vert \geq\frac{\varepsilon}{4}\right\} \\
&  +\Pr\left\{  \sup_{f\in\mathcal{R}}\sup_{\xi\in\Sigma}\left\vert
\hat{\Gamma}_{2}(L_{t},f^{\ast},f^{\ast})\left(  \xi\right)  -{\Gamma}%
_{2}(L_{t},f^{\ast},f^{\ast})\left(  \xi\right)  \right\vert \geq
\frac{\varepsilon}{4}\right\} \\
&  +\Pr\left\{  \sup_{f\in\mathcal{R}}\sup_{\xi\in\Sigma}\left\vert {\Gamma
}_{2}(L_{t},f^{\ast},f^{\ast})\left(  \xi\right)  -{\Gamma}_{2}(L_{t}%
,f,f)\left(  \xi\right)  \right\vert \geq\frac{\varepsilon}{4}\right\} \\
&  +\Pr\left\{  \sup_{f\in\mathcal{R}}\sup_{\xi\in\Sigma}\left\vert {\Gamma
}_{2}(L_{t},f,f)\left(  \xi\right)  -{\Gamma}_{2}(\Delta_{g},f,f)\right\vert
\geq\frac{\varepsilon}{4}\right\} \\
&  \leq P_{1}+P_{2}+P_{3}+P_{4}.
\end{align*}
First, note that for the bilinear form $\hat{\Gamma}_{2}$ we have
\[
\hat{\Gamma}_{2}(L_{t},f,f)\left(  \xi\right)  -\hat{\Gamma}_{2}(L_{t}%
,f^{\ast},f^{\ast})=\hat{\Gamma}_{2}(L_{t},f-f^{\ast},f)\left(  \xi\right)
-\hat{\Gamma}_{2}(L_{t},f^{\ast}-f,f^{\ast})
\]
thus%
\[
\left\vert \hat{\Gamma}_{2}(L_{t},f,f)\left(  \xi\right)  -\hat{\Gamma}%
_{2}(L_{t},f^{\ast},f^{\ast})\right\vert \leq\frac{4}{t^{2}}\left\Vert
f-f^{\ast}\right\Vert _{L^{\infty}}\sup_{f\in\mathcal{R}}\left\Vert
f\right\Vert _{L^{\infty}}.
\]
So we may choose
\[
\delta=\frac{\varepsilon}{16C_{0}}t^{2+\sigma}%
\]
where%
\[
C_{0}=\sup_{F_{x,y}\in\mathcal{R}}\left\Vert F_{x,y}\right\Vert _{\infty}.
\]
With this choice $P_{1}=0.$ \ \ By the same reasoning, also $P_{3}=0.$\newline
Next, we have, by Theorem \ref{extrinsic bias}
\[
\left\vert \sup_{\xi\in\Sigma}{\Gamma}_{2}(L_{t},f,f)\left(  \xi\right)
-{\Gamma}_{2}(\Delta_{g},f,f)\left(  \xi\right)  \right\vert \leq C_{5}t^{1/2}%
\]
where
\[
C_{5}=\sup_{F_{x,y}\in\mathcal{R}}\left\Vert F_{x,y}\right\Vert _{C^{5}\left(
\Sigma\right)  }.
\]
So as long as
\[
t^{1/2}<\frac{\varepsilon}{4C_{5}}%
\]
we have $P_{4}=0.$ \ \ \ We are left to show%
\[
P_{2}=\Pr\left\{  \sup_{f\in\mathcal{R}}\sup_{\xi\in\Sigma}\left\vert
\hat{\Gamma}_{2}(L_{t},f^{\ast},f^{\ast})\left(  \xi\right)  -{\Gamma}%
_{2}(L_{t},f^{\ast},f^{\ast})\left(  \xi\right)  \right\vert \geq
\frac{\varepsilon}{4}\right\}  \rightarrow0.
\]
But by (\ref{gamma2converge}) we have, for each individual $f^{\ast}%
\in\mathcal{G}$%
\[
\Pr\left\{  \sup_{\xi\in\Sigma}\left\vert \hat{\Gamma}_{2}(L_{t},f^{\ast
},f^{\ast})\left(  \xi\right)  -{\Gamma}_{2}(L_{t},f^{\ast},f^{\ast})\left(
\xi\right)  \right\vert \geq\frac{\varepsilon}{4}\right\}  \in O(3d+3),
\]
and the size of the set satisfies
\[
\left\vert \mathcal{G}\right\vert \leq\mathcal{N}\left(  \mathcal{R}%
,\frac{\varepsilon}{16C_{0}}t^{2+\sigma}\right)  .
\]
This in turn is bounded by a polynomial in $\frac{1}{t}$, so we can apply the
Borel Cantelli argument and obtain the result.
\end{proof}

\section{Local PCA and proof of Theorem \ref{Ricciapprox}}

\label{localPCA}

{The goal of this section is to construct a class of test functions that can
be inserted in our construction for $\hat{\Gamma}_{2}(L_{t_{n}},\cdot,\cdot)$
to recover the Ricci curvature as stated in Theorem \ref{Ricciapprox} (in
other words, we will explain how to obtain the functions $f_{n}$ in Theorem
\ref{Ricciapprox}). The key for the construction of these test functions is a
method for estimating a basis of the tangent space to $\Sigma^{d}$ at a given
point $x\in\Sigma$ known as \emph{local PCA} where PCA stands for ``Principal
Component Analysis". The construction that we are about to describe was
developed in \cite{SW12} Section 2.1 and Appendix B, however, for the reader's
convenience we will review the construction without proving any of the
theorems shown in \cite{SW12}. After reviewing the local PCA construction in
\cite{SW12} we will explain how these ideas can be combined with Theorem
\ref{extrinsic variance} to prove Theorem \ref{Ricciapprox}. { }}

\subsection{Estimating an orthonormal basis of the tangent space to a
submanifold at a point}

Let $x\in\Sigma^{d}$ where $\Sigma^{d}$ is again a smooth $d$-dimensional
submanifold of $\mathbb{R}^{N}$ and let $\xi_{1},\xi_{2}\ldots,\xi_{n}$ be
data points on $\Sigma^{d}$. As pointed out in the introduction, suppose that
we fix an embedding $F:\Sigma^{d}\rightarrow\mathbb{R}^{N}$ so that we obtain
a metric $g$ in $\Sigma$ induced by $F$ given in local coordinates by
\begin{align*}
g_{ij}=\langle D_{i}F,D_{j}F\rangle,
\end{align*}
where of course $\langle\cdot,\cdot\rangle$ is the inner product of
$\mathbb{R}^{N}$. We assume that the data points $\{\xi_{j}\}_{j=1}^{d}$ are
uniformly distributed and assuming that we have sufficiently many data points
we expect that many of the points $\{\xi_{j}\}$ will concentrate near $x$.
More precisely, let us fix a positive number $\epsilon>0$ and let
$B_{\epsilon}(x)$ be the geodesic ball of radius $\epsilon$ centered at $x$
with respect to $g$ and let $\{\eta_{1},\ldots,\eta_{N_{\epsilon}}\}$ be the
intersection of $B_{\epsilon}(x)$ with the set $\{\xi_{1},\ldots,\xi_{n}\}$.
For large $n,$ and an analysis similar to the one presented in the previous
section, it is clear that for large $n$, we can choose $\epsilon$ such that
$N_{\epsilon}>\!\!> d$ but $N_{\epsilon}<\!\!<n$. We now choose a function
$\Phi$ satisfying

\begin{itemize}
\item $\Phi$ is supported in $[0,1]$,

\item $\Phi$ is non increasing in $[0,1]$,

\item $\Phi$ is $C^{2}$ on $[0,1]$.
\end{itemize}

One common choice for $\Phi$ is $\Phi(s)=(1-s^{2})\chi_{[0,1]}(s)$. After
choosing $\Phi$, we consider a $N_{\epsilon}\times N_{\epsilon}$ diagonal
matrix $D_{\epsilon}$ with diagonal entries
\begin{align*}
(D_{\epsilon})_{jj}=\sqrt{\Phi\left(  \frac{\|x-\eta_{j}\|}{\epsilon}\right)
},
\end{align*}
for $j=1,\ldots,N_{\epsilon}$. At the same time, consider now the $N\times
N_{\epsilon}$ matrix $X_{\epsilon}$ whose rows are given by
\begin{align*}
X_{\epsilon}=\left[  \eta_{1}-x,\ldots,\eta_{N_{\epsilon}}-x\right]  ,
\end{align*}
and the $N\times N_{\epsilon}$ matrix
\begin{align*}
W_{\epsilon}=X_{\epsilon}D_{\epsilon}.
\end{align*}
The idea of constructing the matrix $W_{\epsilon}$ above is to weight the data
points so that those points that are closer to $x$ are given preference. Next,
the matrix $W_{\epsilon}$ admits a \emph{singular value decomposition} of the
form
\begin{align*}
W_{\epsilon}=U_{\epsilon}\Lambda_{\epsilon}V_{\epsilon}^{T},
\end{align*}
where $U_{\epsilon}$ is a $N\times N_{\epsilon}$ matrix whose columns are
orthonormal in the Hilbert-Schmidt norm and known as the \emph{left singular
vectors} of $W_{\epsilon}$, and $\Lambda_{\epsilon}$ is a diagonal matrix with
non increasing diagonal elements that describe the relative importance of the
vectors. Recall that the Hilbert-Schmidt norm is defined for $N\times d$
matrices by
\begin{align*}
\|A\|_{HS}=\sqrt{\mathrm{tr}(A^{T}A)}.
\end{align*}
We now consider a $N\times d$ matrix $U_{\epsilon}$ with orthonormal columns
given by taking the first $d$ orthonormal singular vectors of $U_{\epsilon}$
(each column of $U_{\epsilon}$ is a singular vector for $W_{\epsilon}$).
Alternatively, or if $d$ is unknown, we can choose the vectors whose weights
in $\Lambda_{\epsilon}$ are bigger than a chosen cutoff value, for example
$1/2$. On a smooth manifold these will agree for large numbers of points. We
will write $U_{\epsilon}$ as
\begin{align}
\label{Oeps}U_{\epsilon}=\left[  \zeta_{1,\epsilon}(x),\ldots,\zeta
_{d,\epsilon}(x)\right]  .
\end{align}
The vectors $\zeta_{j,\epsilon}(x)$ for $j=1,\ldots,d$ form an orthonormal
basis for a $d$-dimensional subspace of $\mathbb{R}^{N}$, in fact, this basis
will serve as an approximation for a basis to the tangent space $T_{x}\Sigma$.
Strictly speaking, the vectors $\zeta_{j,\epsilon}$ depend on the point $x$,
and the data points $\xi_{1},\ldots,\xi_{n}$, but for now we will omit that
dependence. It is important to keep in mind though that the dependence of
$\zeta_{j,\epsilon}$ on these parameters is highly non-linear and in principle
it should create a problem of correlation, however, the empirical theoretic
methods that we have discussed above will allow us to deal with this high correlation

Another important part of the construction is the choice of $\epsilon$, in
fact, as the number of data points tends to infinity, $\epsilon$ will tend to
zero at a definite rate, more precisely, we will choose $\epsilon$ to be
essentially a negative power of $n$ (the number of data points). With these
observations in mind, we can state the main theorem that asserts that the
columns of the matrix $U_{\epsilon}$ defined above converge to a basis of the
tangent space $T_{x}\Sigma^{d}$.

\begin{theorem}
[See Theorem B.1 in \cite{SW12}]\label{localPCAthm} Let $\{\xi_{j}\}_{j=1}%
^{n}$ be a uniformly distributed sample of data points on the embedded
submanifold $\Sigma^{d}\subset\mathbb{R}^{N}$ and let us choose $\epsilon$ by
$\epsilon=\epsilon_{n}=O(n^{-\frac{6}{d+2}})$. There exists a $N\times d$
matrix $\Theta_{\infty}(x)$ whose columns are a basis of $F_{*}\left(
T_{x}\Sigma^{d}\right)  $ and such that with high probability (w.h.p.)
\begin{align*}
\min_{U \in O(d)} \|U^{T}U _{\epsilon_{n}}-\Theta_{\infty}(x)\|_{HS}%
=O(n^{-\frac{3}{d+2}}).
\end{align*}


\end{theorem}

\begin{remark}
\emph{{ The estimate with high probability of $\Theta_{\infty}(x)$ by
$U_{\epsilon_{n}}$ implies that $U_{\epsilon_{n}}$ converges to $\Theta
_{\infty}(x)$ almost surely. } }
\end{remark}

In the next section we show how to use Theorem \ref{localPCAthm} to prove
Theorem \ref{Ricciapprox}.

\subsection{Proof Of Theorem \ref{Ricciapprox}}

We start by observing that if one knows a tangent vector $\eta$ to $\Sigma$ to
a point $x$, then one can construct a test function such that the iterated
Carr\'{e} du Champ applied to that function is precisely $\mathrm{Ric}%
_{x}(\eta,\eta)$.

\begin{proposition}
\label{propetaf} Let $F:\Sigma^{d}\rightarrow\mathbb{R}^{N}$ be an embedding
of $\Sigma$ in $\mathbb{R}^{N}$ and let $f:\mathbb{R}^{N}\rightarrow
\mathbb{R}$ be given by $f(z)=\langle z,\eta\rangle$ where $\eta\in
\mathbb{R}^{N}$ is fixed. If $g$ is the metric induced by the embedding
$F:\Sigma\rightarrow\mathbb{R}^{N}$ and if we fix a point $x\in\Sigma$ we
obtain
\begin{align*}
\Gamma_{2}(\Delta_{g},f,f)(x)=\mathrm{Ric}_{x}(\eta^{T},\eta^{T})+\left(
(\eta)^{\perp}\right)  ^{2}\|\mathrm{II}_{\Sigma}\|^{2}_{\Sigma,x},
\end{align*}
where $\eta^{T}$ and $\eta^{\perp}$ are the orthogonal projections of $\eta$
onto $F_{*}\left(  T_{x}\Sigma\right)  $ and $\left(  F_{*}\left(  T_{x}%
\Sigma\right)  \right)  ^{\perp}$ respectively and $\mathrm{II}_{\Sigma}$ is
the second fundamental form of $\Sigma$. In particular, if $\eta\in
F_{*}\left(  T_{x}\Sigma\right)  $ we have
\begin{align*}
\Gamma_{2}(\Delta_{g},f,f)(x)=\mathrm{Ric}_{x}(\eta,\eta).
\end{align*}

\end{proposition}

\begin{proof}
For simplicity let us only prove the codimension 1 case. By \eqref{bochner} we
have
\begin{align*}
\Gamma_{2}(\Delta_{g},f,f)(x)=\mathrm{Ric}_{x}(\nabla_{\Sigma} f,\nabla
_{\Sigma} f)+\left\|  \nabla^{2}_{\Sigma}f\right\|  ^{2}_{\Sigma,x}.
\end{align*}
Observe that $(\nabla_{\Sigma}f)_{x}$ is obtained by taking the projection of
$Df$ (ambient derivative) onto $F_{*}\left(  T_{x}\Sigma\right)  $ and
therefore $(\nabla_{\Sigma}f)_{x}=\eta^{T}$. Observe now that we have
\begin{align*}
\nabla^{2}_{\Sigma}f=D^{2}f-(Df)^{\perp}\mathrm{II}_{\Sigma}=-(Df)^{\perp
}\mathrm{II}_{\Sigma}.
\end{align*}
The proposition follows.
\end{proof}

Next, we show how the approximate Ricci is constructed from the basis. For
each of the vectors $\zeta_{j}(x)$ determined by the PCA, consider the linear
function
\[
f_{n,j}(z)=\langle z,\zeta_{j}\rangle
\]
and then define
\begin{equation}
\hat{R}_{i,j}=\hat{\Gamma}_{2}(L_{t_{n}},f_{n,i},f_{n,j}). \label{pcaricci}%
\end{equation}
For any vector (in the tangent space, approximate tangent space, or neither)
we can define the approximate Ricci curvature of $\eta$ by projecting $\eta$
onto the vectors $\zeta_{j}$ and summing the linear combination as follows.
Projecting onto the approximate basis and splitting the vector, let
\[
\eta=\eta^{A}+\eta^{\perp}=\eta^{j}\zeta_{j}+\eta^{\perp}%
\]
and define
\[
\hat{\mathrm{Ric}}(\eta,\eta)=\hat{R}_{i,j}\eta^{i}\eta^{j}.
\]

\begin{proof}
[Proof of Theorem \ref{Ricciapprox}]For a given vector $\eta\in T_{x}M$ define
the function $f(z)=\langle z,\eta\rangle$ as above. Let
\[
\eta=\eta^{A}+\eta^{\perp}%
\]
and define $f^{A}(z)=\langle z,\eta^{A}\rangle$ and $f^{\perp}(z)=\langle
z,\eta^{\perp}\rangle$ \ so that
\begin{equation}
\hat{\mathrm{Ric}}_{x}(\eta,\eta)=\hat{\Gamma}_{2}(L_{t_{n}},f^{A},f^{A}).
\label{aRicf}%
\end{equation}

Now we compute the difference of the actual Ricci and the approximate Ricci,
using Proposition \ref{propetaf}, and (\ref{aRicf})
\begin{align}
\mathrm{Ric}_{x}(\eta,\eta)-\hat{\mathrm{Ric}}_{x}(\eta,\eta)  &  =\Gamma
_{2}(\Delta_{g},f,f)-\hat{\Gamma}_{2}(L_{t_{n}},f^{A},f^{A}))\\
&  =\Gamma_{2}(\Delta_{g},f,f)-\Gamma_{2}(\Delta_{g},f^{A},f^{A})\\
&  +\Gamma_{2}(\Delta_{g},f^{A},f^{A})-\Gamma_{2}(L_{t_{n}},f^{A},f^{A})\\
&  +\Gamma_{2}(L_{t_{n}},f^{A},f^{A})-\hat{\Gamma}_{2}(L_{t_{n}},f^{A},f^{A})
\end{align}
Observe that clearly, all functions $f^{A}$ are in the class $\mathscr{L}_{M}$
for some fixed $M>0$ and therefore
\[
\left\vert \left(  \hat{\Gamma}_{2}(L_{t_{n}},f^{A},f^{A})-\Gamma_{2}%
(L_{t_{n}},f^{A},f^{A})\right)  \right\vert \leq\sup_{f\in\mathscr{L}_{M}}%
\sup_{\xi\in\Sigma}\left\vert \hat{\Gamma}_{2}(L_{t_{n}},f,f)(\xi)-\Gamma
_{2}(L_{t_{n}},f,f)(\xi)\right\vert ,
\]
and from Theorem \ref{extrinsic variance}, with the given choice of scale
$t_{n},$ we have that
\[
\left\vert \Gamma_{2}(L_{t_{n}},f^{A},f^{A})-\hat{\Gamma}_{2}(L_{t_{n}}%
,f^{A},f^{A})\right\vert \overset{\mathrm{a.s.}}{\longrightarrow}0.
\]
Similarly, it follows from Theorem \ref{expansion bias} that
\[
\left\vert \Gamma_{2}(\Delta_{g},f^{A},f^{A})-\Gamma_{2}(L_{t_{n}},f^{A}%
,f^{A})\right\vert \overset{}{\longrightarrow}0.
\]
Now certainly, as the PCA is choosing linear functions that recover the
tangent space with high probability in the limit, the linear functions $f^{A}$
converge with high probability uniformly in all orders to the function $f$ on
the manifold. \ It follows that the term
\[
\left\vert \Gamma_{2}(\Delta_{g},f,f)-\Gamma_{2}(\Delta_{g},f^{A}%
,f^{A})\right\vert \overset{\mathrm{a.s}}{\longrightarrow}0.
\]
Combining the above three limits we have
\[
\left\vert \mathrm{Ric}_{x}(\eta,\eta)-\hat{\mathrm{Ric}}_{x}(\eta
,\eta)\right\vert \overset{\mathrm{a.s}}{\longrightarrow}0.
\]

\end{proof}

\subsection{Dimension estimation} \label{five3}
 The problem of estimating the dimension of the underlying
submanifold assuming the manifold hypothesis is a fascinating subject in
itself and there is a large number of estimators for the intrinsic dimension
in the manifold learning problem that have been proposed in the literature.
Some of the early approaches for dimension estimation were based on PCA or also the application of the Vapnik-Chervonenkis classes of
sets of separating hyperplanes \cite{Vap98}. These methods lose effectiveness
when applied to relatively highly nonlinear problems or in the presence of
noise. There is a wide variety of methods that have been proposed in order to
overcome these difficulties, and many of them use of techniques in machine
learning that are nowadays well known, for example methods based on suitable
maximum likelihood estimators applied to distances to the {$k$ Nearest
Neighbors (kNN)}. See for example \cite{RLR11,RLCC12} where many of these
ideas are discussed in detail. There are also fractal-based methods, and
methods based on the concept of ISOMAP, which consists in estimating the
distance of points nearby by the geodesic distance whenever possible and, and
estimating the distance between points that are far apart by means of the
shortest path in the graph that is used to approximate the submanifold. See
for example \cite{BusSomm98}, \cite{CaVin02} and \cite{WM08} and the
references therein. See also the discussion in \cite[page 1073]{SW12}, the
references therein and \cite{little11}.

\bibliographystyle{plain}
\bibliography{coarse_ricci}

\end{document}